\documentclass[english]{article}
\usepackage[T1]{fontenc}
\usepackage[latin9]{inputenc}
\usepackage{geometry}
\usepackage{amsthm}
\usepackage{amsmath}
\usepackage{amssymb}
\usepackage{esint}

\makeatletter
\numberwithin{equation}{section}
\theoremstyle{plain}
\newtheorem{thm}{\protect\theoremname}[section]
  \theoremstyle{definition}
  \newtheorem{defn}[thm]{\protect\definitionname}
  \theoremstyle{remark}
  \newtheorem{rem}[thm]{\protect\remarkname}
  \theoremstyle{plain}
  \newtheorem{assumption}[thm]{\protect\assumptionname}
  \theoremstyle{plain}
  \newtheorem{lem}[thm]{\protect\lemmaname}
  \theoremstyle{plain}
  \newtheorem{cor}[thm]{\protect\corollaryname}

\makeatother

\usepackage{babel}
  \providecommand{\assumptionname}{Assumption}
  \providecommand{\corollaryname}{Corollary}
  \providecommand{\definitionname}{Definition}
  \providecommand{\lemmaname}{Lemma}
  \providecommand{\remarkname}{Remark}
\providecommand{\theoremname}{Theorem}

\begin{document}

\title{Stochastic dominance-constrained Markov decision processes}

\author{William B. Haskell%
\thanks{William Haskell is a visiting assistant professor in the Department
of Industrial and Systems Engineering at the University of Southern
California.%
}\hspace{0.1in}and Rahul Jain%
\thanks{Rahul Jain is an assistant professor in the Departments of Electrical
Engineering and Industrial and Systems Engineering at the University
of Southern California. This research is supported by the Air Force
Office of Scientific Research and the Office of Naval Research.%
}}
\maketitle
\begin{abstract}
We are interested in risk constraints for infinite horizon discrete
time Markov decision processes (MDPs). Starting with average reward
MDPs, we show that increasing concave stochastic dominance constraints
on the empirical distribution of reward lead to linear constraints
on occupation measures. The optimal policy for the resulting class
of dominance-constrained MDPs is obtained by solving a linear program.
We compute the dual of this linear program to obtain average dynamic
programming optimality equations that reflect the dominance constraint.
In particular, a new pricing term appears in the optimality equations
corresponding to the dominance constraint. We show that many types
of stochastic orders can be used in place of the increasing concave
stochastic order. We also carry out a parallel development for discounted
reward MDPs with stochastic dominance constraints. The paper concludes
with a portfolio optimization example.
\end{abstract}

\section{Introduction}

Markov decision processes (MDPs) are a natural and powerful framework
for stochastic control problems. In the present paper, we take up
the issue of risk constraints in MDPs. Convex analytic methods for
MDPs have been successful at handling many types of constraints. Our
specific goal is to find and study risk constraints for MDPs that
are amenable to convex analytic formulation. It turns out that stochastic
dominance constraints are natural risk constraints for MDPs.

Convex analytic methods are well studied for Markov decision processes.
The linear programming approach for MDPs is pioneered in \cite{Kal80},
and an early survey is found in \cite{ABF+93}. The main idea is that
some MDPs can be written as convex optimization problems in terms
of appropriate occupation measures. \cite{Bor88,HG98,Bor02,HL02}
discuss a rigorous theory of convex optimization for MDPs with general
Borel state and action spaces. Detailed monographs on Markov decision
processes are found in \cite{HL96,HL99,Put05}. Constrained MDPs can
naturally be embedded in this framework. Constrained discounted MDPs
are explored in \cite{FS95,FS96}. \cite{Alt99} is a substantial
monograph on constrained MDPs. Constrained discounted MDPs in Borel
spaces are analyzed in \cite{HG00}, and constrained average cost
MDPs in Borel spaces are developed in \cite{HGL03}. Infinite dimensional
linear programming plays a fundamental role in both \cite{HG00,HGL03},
and the theory of infinite dimensional linear programming is developed
in \cite{AN87}. The special case of constraints on expected utility
in discounted MDPs is considered in \cite{KKY06}. MDPs with expected
constraints and pathwise constraints, also called hard constraints,
are considered in \cite{MH10} using convex analytic methods. An inventory
system is detailed to motivate the theoretical results.

Policies in MDPs induce Markov chains. Typically, policies are evaluated
with respect to some measure of expected reward, such as long-run
average reward or discounted reward. The variation/spread/dispersion
of policies is also critical to their evaluation. Given two policies
with equal expected performance, we would prefer the one with smaller
variation in some sense. Consider a discounted portfolio optimization
problem, for example. The expected discounted reward of an investment
policy is a key performance measure; the downside variation of an
investment policy is also a key performance measure. When rewards
and costs are involved, the variation of a policy can also be called
its risk.

Risk management for MDPs has been considered from many perspectives
in the literature. \cite{Filar1989} includes penalties for the variance
of rewards in MDPs. The optimal policy is obtained by solving a nonlinear
programming problem in occupation measures. In \cite{Sobel1994},
the mean-variance trade-off in MDPs is further explored in a Pareto-optimality
sense. The conditional value-at-risk of the total cost in a finite
horizon MDPs is constrained in \cite{BJ10}. It is argued that convex
analytic methods do not apply to this problem type and an offline
iterative algorithm is employed to solve for the optimal policy. \cite{Rus10}
develops Markov risk measures for finite horizon and infinite horizon
discounted MDPs. Dynamic programming equations are derived that reflect
the risk aversion, and policy iteration is shown to solve the infinite
horizon problem.

Our notion of risk constrained MDPs differs from this literature survey.
We are interested in the empirical distribution of reward, rather
than in its expectation, variance, or other summary statistics. Our
approach is based on stochastic orders, which are partial orders on
the space of random variables, see \cite{Muller2002,Shaked2007} for
extensive monographs on stochastic orders. \cite{Dentcheva2003,Dentcheva2004}
use the increasing concave stochastic order to define stochastic dominance
constraints in single stage stochastic optimization. The increasing
concave stochastic order is notable for its connection to risk-averse
decision makers, i.e. it captures the preferences of all risk-averse
decision makers. A benchmark random variable is introduced, and a
concave random variable-valued mapping is constrained to dominate
the benchmark in the increasing concave stochastic order. It is shown
that increasing concave functions are the Lagrange multipliers of
the dominance constraints. The dual problem is a search over a certain
class of increasing concave functions, interpreted as utility functions,
and strong duality is established. Stochastic dominance constraints
are applied to finite horizon stochastic programming problems with
linear system dynamics in \cite{Dentcheva2008}. Specifically, a stochastic
dominance constraint is placed on a vector of state and action dependent
reward functions across the finite planning horizon. The Lagrange
multipliers of this dynamic stochastic dominance constraint are again
determined to be increasing concave functions, and strong duality
holds. In contrast, we place a stochastic dominance constraint on
the empirical distribution of reward in infinite horizon MDPs. We
argue that this type of constraint comprehensively accounts for the
variation in policies in MDPs.

We make two main contributions in this paper. First, we show how to
formulate stochastic dominance constraints for long-run average reward
maximizing MDPs. More immediately, we show that stochastic dominance
constrained MDPs can be solved via linear programming over occupation
measures. Our model is more general than \cite{Dentcheva2008} because
it allows for an arbitrary transition kernel and is also infinite
horizon. Also, our model is more computationally tractable than the
stochastic programming model in \cite{Dentcheva2008} because it leads
to linear programs. Second, we apply infinite-dimensional linear programming
duality to gain more insight: the resulting duals are similar to the
linear programming form of the average reward dynamic programming
optimality equations. However, new decision variables corresponding
to the stochastic dominance constraint appear in an intuitive way.
Specifically, the new decision variables are increasing concave functions
that price rewards. This observation parallels the results in \cite{Dentcheva2003,Dentcheva2004,Dentcheva2009}
and is natural because our stochastic dominance constraints are defined
in terms of increasing concave functions. The upcoming dual problems
are themselves linear programs, unlike the dual problems in \cite{Dentcheva2003,Dentcheva2004,Dentcheva2009}
which are general infinite-dimensional convex optimization problems.

This paper is organized as follows. In section 2, we consider stochastic
dominance constraints for long-run average reward maximizing MDPs.
In section 3 we formulate this problem as a static optimization problem,
in fact a linear programming problem, in a space of occupation measures.
Section 4 develops the dual for this problem using infinite dimensional
linear programming duality, and reveals the form of the Lagrange multipliers.
In section 5, we discuss a number of immediate variations and extensions,
especially the drastically simpler development on finite state and
action spaces. We illustrate our method in section 6 with a portfolio
optimization example, and then conclude the paper in section 7.

\section{MDPs and stochastic dominance}

The first subsection presents a general model for average reward MDPs,
and the second explains how to apply stochastic dominance constraints.

\subsection{Average reward MDPs}

A typical representation of a discrete time MDP is the 5-tuple
\[
\left(S,\, A,\,\left\{ A\left(s\right)\mbox{ : }s\in S\right\} ,\, Q,\, r\right).
\]
The state space $S$ and the action space $A$ are Borel spaces, subsets
of complete and separable metric spaces, with corresponding Borel
$\sigma-$algebras $\mathcal{B}\left(S\right)$ and $\mathcal{B}\left(A\right)$.
We define $\mathcal{P}\left(S\right)$ to be the space of probability
measures over $S$ with respect to $\mathcal{B}\left(S\right)$, and
we define $\mathcal{P}\left(A\right)$ analogously. For each state
$s\in S$, the set $A\left(s\right)\subset A$ is a measurable set
in $\mathcal{B}\left(A\right)$ and indicates the set of feasible
actions available in state $s$. The set of feasible state-action
pairs is written

\[
K=\left\{ \left(s,a\right)\in S\times A\mbox{ : }a\in A\left(s\right)\right\} ,
\]
and $K$ is assumed to be closed in $S\times A$. The transition law
$Q$ governs the system evolution. Explicitly, $Q\left(B\mbox{ | }s,\, a\right)$
for $B\in\mathcal{B}\left(S\right)$ is the probability of visiting
the set $B$ given the state-action pair $\left(s,a\right)$. Finally,
$r\mbox{ : }K\rightarrow\mathbb{R}$ is a measurable reward function
that depends on state-action pairs.

We now describe two classes of policies for MDPs. Let $H_{t}$ be
the set of \textit{histories} at time $t$, $H_{0}=S$, $H_{1}=K\times S$,
and $H_{t}=K^{t}\times S$ for all $t\geq2$. A specific history $h_{t}\in H_{t}$
records the state-action pairs visited at times $0,1,\ldots,t-1$
and the current state $s_{t}$. Define $\Pi$ to be the set of all
\textit{history-dependent randomized policies}: collections of mappings
$\pi_{t}\mbox{ : }H_{t}\rightarrow\mathcal{P}\left(A\right)$ for
all $t\geq0$. Given a history $h_{t}\in H_{t}$ and a set $B\in\mathcal{B}\left(A\right)$,
$\pi\left(B\mbox{ | }h_{t}\right)$ is the probability of selecting
an action in $B$. Define $\Phi$ to be the class of \textit{stationary randomized Markov policies}:
mappings $\phi\mbox{ : }S\rightarrow\mathcal{P}\left(A\right)$ which
only depend on history through the current state. For a given state
$s\in S$ and a set $B\in\mathcal{B}\left(A\right)$, $\phi\left(B\mbox{ | }s\right)$
is the probability of choosing an action in $B$. The class $\Phi$
will be viewed as a subset of $\Pi$. We explicitly assume that both
$\Pi$ and $\Phi$ only include feasible policies that respect the
constraints $K$.

The state and action at time $t$ are denoted $s_{t}$ and $a_{t}$,
respectively. Any policy $\pi\in\Pi$ and initial distribution $\nu\in\mathcal{P}\left(S\right)$
determines a probability measure $P_{\nu}^{\pi}$ and stochastic process
$\left\{ \left(s_{t},a_{t}\right),\, t\geq0\right\} $ defined on
a measurable space $\left(\Omega,\mathcal{F}\right)$. The expectation
operator with respect to $P_{\nu}^{\pi}$ is denoted $\mathbb{E}_{\nu}^{\pi}\left[\cdot\right]$.
Consider the \textit{long-run expected average reward}

\[
R\left(\pi,\nu\right)=\liminf_{T\rightarrow\infty}\frac{1}{T}\mathbb{E}_{\nu}^{\pi}\left[\sum_{t=0}^{T-1}r\left(s_{t},a_{t}\right)\right].
\]
The classic \textit{long-run expected average reward maximization problem}
is

\begin{align}
\sup\hspace{0.2in} & R\left(\pi,\nu\right)\label{CLASSIC}\\
\mbox{s.t.}\hspace{0.2in} & \pi\in\Pi.\label{CLASSIC-1}
\end{align}
It is known that a stationary policy in $\Phi$ is optimal for problem
(\ref{CLASSIC}) - (\ref{CLASSIC-1}) under suitable conditions (this
result is found in \cite{Put05} for finite and countable state spaces,
and \cite{HL96,HL99} for general Borel state and action spaces).

\subsection{Stochastic dominance}

Now we will motivate and formalize stochastic dominance constraints
for problem (\ref{CLASSIC}) - (\ref{CLASSIC-1}). To begin, let $z\mbox{ : }K\rightarrow\mathbb{R}$
be another measurable reward function, possibly different from $r$.
A risk-averse decision maker with an increasing concave utility function
$u\mbox{ : }\mathbb{R}\rightarrow\mathbb{R}$ would be interested
in maximizing his long-run average expected utility
\[
\liminf_{T\rightarrow\infty}\frac{1}{T}\mathbb{E}_{\nu}^{\pi}\left[\sum_{t=0}^{T-1}u\left(z\left(s_{t},a_{t}\right)\right)\right].
\]
However, it is difficult to choose one utility function to represent
a risk-averse decision maker without considerable information. We
will use the increasing concave order to express a continuum of risk
preferences in MDPs.
\begin{defn}
For random variables $X,\, Y\in\mathbb{R}$, $X$ \textit{dominates}
$Y$ in the \textit{increasing concave stochastic order}, written
$X\geq_{icv}Y$, if $\mathbb{E}\left[u\left(X\right)\right]\geq\mathbb{E}\left[u\left(Y\right)\right]$
for all increasing concave functions $u\mbox{ : }\mathbb{R}\rightarrow\mathbb{R}$
such that both expectations exist.
\end{defn}
Let $\mathcal{C}\left(\mathbb{R}\right)$ be the set of all continuous
functions $f\mbox{ : }\mathbb{R}\rightarrow\mathbb{R}$. Let $\mathcal{U}\left(\mathbb{R}\right)\subset\mathcal{C}\left(\mathbb{R}\right)$
be the set of all increasing concave functions $u\mbox{ : }\mathbb{R}\rightarrow\mathbb{R}$
such that
\[
\lim_{x\rightarrow\infty}u\left(x\right)=0
\]
and
\[
u\left(x\right)=u\left(x_{0}\right)+\kappa\left(x-x_{0}\right)
\]
for all $x\leq x_{0}$ for some $\kappa>0$ and $x_{0}\in\mathbb{R}$
(the choices of $\kappa$ and $x_{0}$ differ among $u$). The second
condition just means that all $u\in\mathcal{U}\left(\mathbb{R}\right)$
become linear as $x\rightarrow-\infty$. By construction, functions
$u\in\mathcal{U}\left(\mathbb{R}\right)$ are bounded from above by
zero. We will use the set $\mathcal{U}\left(\mathbb{R}\right)$ to
characterize $X\geq_{icv}Y$.

Now define $\left(x\right)_{-}\triangleq\min\left\{ x,0\right\} $.
We note that any function in $\mathcal{U}\left(\mathbb{R}\right)$
can be written in terms of the family $\left\{ \left(x-\eta\right)_{-}\mbox{ : }\eta\in\mathbb{R}\right\} $.
To understand this result, choose $u\in\mathcal{U}\left(\mathbb{R}\right)$
and a finite set of points $\left\{ x_{1},\ldots,x_{j}\right\} $.
By concavity, there exist $a_{i}\in\mathbb{R}$ such that $a_{i}\left(x-x_{i}\right)+u\left(x_{i}\right)\geq u\left(x\right)$
for all $x\in\mathbb{R}$ and for all $i=1,\ldots,j$. Each linear
function $a_{i}\left(x-x_{i}\right)+u\left(x_{i}\right)$ is a global
over-estimator of $u$. The piecewise linear increasing concave function
\[
\min_{i=1,\ldots,j}\left\{ a_{i}\left(x-x_{i}\right)+u\left(x_{i}\right)\right\} 
\]
is also a global over-estimator of $u$, and certainly
\[
u\left(x\right)\leq\min_{i=1,\ldots,j}\left\{ a_{i}\left(x-x_{i}\right)+u\left(x_{i}\right)\right\} \leq a_{i}\left(x-x_{i}\right)+u\left(x_{i}\right)
\]
for all $i=1,\ldots,j$ and $x\in\mathbb{R}$. As the number of sample
points $j$ increases, the polyhedral concave function $\min_{i=1,\ldots,j}\left\{ a_{i}\left(x-x_{i}\right)+u\left(x_{i}\right)\right\} $
becomes a better approximation of $u$. We realize that the function
$\min_{i=1,\ldots,j}\left\{ a_{i}\left(x-x_{i}\right)+u\left(x_{i}\right)\right\} $
is equal to a finite sum of nonnegative scalar multiples of functions
from $\left\{ \left(x-\eta\right)_{-}\mbox{ : }\eta\in\mathbb{R}\right\} $.
It follows that the relation $X\geq_{icv}Y$ is equivalent to $\mathbb{E}\left[\left(X-\eta\right)_{-}\right]\geq\mathbb{E}\left[\left(Y-\eta\right)_{-}\right]$
for all $\eta\in\mathbb{R}$. When the support of $Y$ is contained
in a compact interval $\left[a,b\right]$, the condition $\mathbb{E}\left[\left(X-\eta\right)_{-}\right]\geq\mathbb{E}\left[\left(Y-\eta\right)_{-}\right]$
for all $\eta\in\left[a,b\right]$ is sufficient for $X\geq_{icv}Y$.

From now on, let $Y$ be a fixed reference random variable on $\mathbb{R}$
to benchmark the empirical distribution of reward $z$. We assume
that $Y$ has support in a compact interval $\left[a,b\right]$ throughout
the rest of this paper. Define
\[
Z_{\eta}\left(\pi,\nu\right)\triangleq\liminf_{T\rightarrow\infty}\frac{1}{T}\mathbb{E}_{\nu}^{\pi}\left[\sum_{t=0}^{T-1}\left(z\left(s_{t},a_{t}\right)-\eta\right)_{-}\right]
\]
to be the long-run expected average shortfall in $z$ at level $\eta$.
We propose the class of \textit{stochastic dominance-constrained MDPs}:

\begin{align}
\sup\hspace{0.2in} & R\left(\pi,\nu\right)\label{DOMINANCE}\\
\mbox{s.t.}\hspace{0.2in} & Z_{\eta}\left(\pi,\nu\right)\geq\mathbb{E}\left[\left(Y-\eta\right)_{-}\right], & \forall\eta\in\left[a,b\right],\label{DOMINANCE-1}\\
 & \pi\in\Pi.\label{DOMINANCE-2}
\end{align}
For emphasis, we index $\eta$ over the compact set $\left[a,b\right]$
in (\ref{DOMINANCE-1}). Allowing $\eta$ to range over all $\mathbb{R}$
would lead to major technical difficulties, as first observed in \cite{Dentcheva2003,Dentcheva2004}.

Constraint (\ref{DOMINANCE-2}) is a continuum of constraints on the
long-run expected average shortfall of the policy $\pi$ for all $\eta\in\left[a,b\right]$.
We will approach problem (\ref{DOMINANCE}) - (\ref{DOMINANCE-2})
by casting it in the space of long-run average occupation measures.
Then we will see that constraint (\ref{DOMINANCE-1}) is equivalent
to a stochastic dominance constraint on the empirical distribution
of rewards $z$, namely 
\[
\lim_{T\rightarrow\infty}\frac{1}{T}\sum_{t=0}^{T-1}z\left(s_{t},a_{t}\right)\geq_{icv}Y.
\]
To be clear, $\lim_{T\rightarrow\infty}\frac{1}{T}\sum_{t=0}^{T-1}z\left(s_{t},a_{t}\right)$
indicates a random variable on $\mathbb{R}$, not the long-run average
of $z\left(s_{t},a_{t}\right)$.

We can denote the feasible region of problem (\ref{DOMINANCE}) -
(\ref{DOMINANCE-2}) succinctly as

\[
\Delta\triangleq\left\{ \left(\pi,\nu\right)\in\Pi\times\mathcal{P}\left(S\right)\mbox{ : }R\left(\pi,\nu\right)>-\infty\mbox{ and }Z_{\eta}\left(\pi,\nu\right)\geq\mathbb{E}\left[\left(Y-\eta\right)_{-}\right]\mbox{ for all }\eta\in\left[a,b\right]\right\} ,
\]
allowing problem (\ref{DOMINANCE}) - (\ref{DOMINANCE-2}) to be written
as
\[
\rho^{*}\triangleq\sup\left\{ R\left(\pi,\nu\right)\mbox{ : }\left(\pi,\nu\right)\in\Delta\right\} ,
\]
where $\rho^{*}$ is the optimal value.
\begin{rem}
We focus on the average reward case in this paper. The extension to
the average cost case is immediate. Let $c\mbox{ : }S\times A\rightarrow\mathbb{R}$
be a measurable cost function. The long-run expected average cost
is
\[
C\left(\pi,\nu\right)=\limsup_{T\rightarrow\infty}\frac{1}{T}\mathbb{E}_{\nu}^{\pi}\left[\sum_{t=0}^{T-1}c\left(s_{t},a_{t}\right)\right].
\]
Similarly, let $\mathfrak{z}\mbox{ : }S\times A\rightarrow\mathbb{R}$
be another measurable cost function that possibly differs from $c$.
Since $\mathfrak{z}$ represents costs, we want the empirical distribution
of $\mathfrak{z}$ to be ``small'' in a stochastic sense. For costs,
it is logical to use the increasing convex order rather than the increasing
concave order. For random variables $X,\, Y\in\mathbb{R}$, $X$ \textit{dominates}
$Y$ in the \textit{increasing convex stochastic order}, written
$X\geq_{icx}Y$, if $\mathbb{E}\left[f\left(X\right)\right]\geq\mathbb{E}\left[f\left(Y\right)\right]$
for all increasing convex functions $f\mbox{ : }\mathbb{R}\rightarrow\mathbb{R}$
such that both expectations exist. Define $\left(x\right)_{+}\triangleq\max\left\{ x,0\right\} $,
and recall that the relation $X\geq_{icx}Y$ is equivalent to $\mathbb{E}\left[\left(X-\eta\right)_{+}\right]\geq\mathbb{E}\left[\left(Y-\eta\right)_{+}\right]$
for all $\eta\in\mathbb{R}$. When the support of $Y$ is contained
in an interval $\left[a,b\right]$, the relation $X\geq_{icx}Y$ is
equivalent to $\mathbb{E}\left[\left(X-\eta\right)_{+}\right]\geq\mathbb{E}\left[\left(Y-\eta\right)_{+}\right]$
for all $\eta\in\left[a,b\right]$.

Momentarily, let $Y$ be a benchmark random variable that we require
to dominate the empirical distribution of $\mathfrak{z}$. Define
\[
\mathfrak{Z}_{\eta}\left(\pi,\nu\right)\triangleq\limsup_{T\rightarrow\infty}\frac{1}{T}\mathbb{E}_{\nu}^{\pi}\left[\sum_{t=0}^{T-1}\left(\mathfrak{z}\left(s_{t},a_{t}\right)-\eta\right)_{+}\right]
\]
for all $\eta\in\left[a,b\right]$. We obtain the cost minimization
problem
\begin{align*}
\inf\hspace{0.2in} & C\left(\pi,\nu\right)\\
\mbox{s.t.}\hspace{0.2in} & \mathfrak{Z}_{\eta}\left(\pi,\nu\right)\leq\mathbb{E}\left[\left(Y-\eta\right)_{+}\right], & \forall\eta\in\left[a,b\right],\\
 & \pi\in\Pi.
\end{align*}
The upcoming results of this paper all have immediate analogs for
the average cost case.
\end{rem}

\section{A linear programming formulation}

This section develops problem (\ref{DOMINANCE}) - (\ref{DOMINANCE-2})
as an infinite dimensional linear program. First, we discuss occupation
measures on the set $K$. Occupation measures on $K$ can be interpreted
as the long-run average expected number of visits of a stochastic
process $\left\{ \left(s_{t},a_{t}\right),\, t\geq0\right\} $ to
each state-action pair. Next, we argue that a stationary policy in
$\Phi$ is optimal for problem (\ref{DOMINANCE}) - (\ref{DOMINANCE-2}).
It will follow that the functions $R\left(\phi,\nu\right)$ and $Z_{\eta}\left(\phi,\nu\right)$
can be written as linear functions of the occupation measure corresponding
to $\phi$ and $\nu$. These linear functions give us the desired
linear program.

To proceed, we recall several well known results in convex analytic
methods for MDPs. We will use $\mu$ to denote probability measures
on $K$, and the set of all probability measures on $K$ is denoted
$\mathcal{P}\left(K\right)$. Probability measures on $K$ can be
equivalently viewed as probability measures on all of $S\times A$
with all mass concentrated on $K$, $\mu\left(K\right)=1$. For any
$\mu\in\mathcal{P}\left(K\right)$, the \textit{marginal} of $\mu$
on $S$ is the probability measure $\hat{\mu}\in\mathcal{P}\left(S\right)$
defined by $\hat{\mu}\left(B\right)=\mu\left(B\times A\right)$ for
all $B\in\mathcal{B}\left(S\right)$.

The following two well known facts are ubiquitous in the literature
on convex analytic methods for MDPs (see \cite{Dynkin_1979} for example).
First, if $\mu$ is a probability measure on $K$, then there exists
a stationary randomized Markov policy $\phi\in\Phi$ such that $\mu$
can be \textit{disintegrated} as $\mu=\hat{\mu}\cdot\phi$ where
$\hat{\mu}$ is the marginal of $\mu$. Specifically, $\mu=\hat{\mu}\cdot\phi$
is defined by
\[
\mu\left(B\times C\right)=\int_{B}\phi\left(C\mbox{ | }s\right)\hat{\mu}\left(ds\right)
\]
for all $B\in\mathcal{B}\left(S\right)$ and $C\in\mathcal{B}\left(A\right)$.
Second, for each $\phi\in\Phi$ and $\nu\in\mathcal{P}\left(S\right)$,
the probability measure $\mu=\nu\cdot\phi$ on $S\times A$ satisfies
$\mu\left(K\right)=1$ and $\hat{\mu}=\nu$. Specifically, $\mu=\nu\cdot\phi$
is defined by
\[
\mu\left(B\times C\right)=\int_{B}\phi\left(C\mbox{ | }s\right)\nu\left(ds\right)
\]
for all $B\in\mathcal{B}\left(S\right)$ and $C\in\mathcal{B}\left(A\right)$.

We can integrate measurable functions $f$ on $K$ with respect to
measures $\mu\in\mathcal{P}\left(K\right)$. Define
\[
\langle\mu,f\rangle\triangleq\int_{K}f\left(s,a\right)\mu\left(d\left(s,a\right)\right)
\]
as the integral of $f$ over state-action pairs $\left(s,a\right)\in K$
with respect to $\mu$. Then
\[
\langle\mu,\, r\rangle=\int_{K}r\left(s,a\right)\mu\left(d\left(s,a\right)\right)
\]
is the expected reward with respect to the probability measure $\mu$
and

\[
\langle\mu,\,\left(z-\eta\right)_{-}\rangle=\int_{K}\left(z\left(s,a\right)-\eta\right)_{-}\mu\left(d\left(s,a\right)\right)
\]
is the expected shortfall in $z$ at level $\eta$ with respect to
the probability measure $\mu$.

We need to restrict to a certain class of probability measures. For
notational convenience, define $r\left(s,\phi\right)\triangleq\int_{A}r\left(s,a\right)\phi\left(da\mbox{ | }s\right)$
and $Q\left(\cdot\mbox{ | }s,\,\phi\right)\triangleq\int_{A}Q\left(\cdot\mbox{ | }s,\, a\right)\phi\left(da\mbox{ | }s\right)$.
\begin{defn}
\cite[Definition 3.4]{HGL03} A probability measure $\mu=\hat{\mu}\cdot\phi$
is called \textit{stable} if 
\[
\langle\mu,r\rangle=\int r\left(s,a\right)\mu\left(d\left(s,a\right)\right)>-\infty
\]
and the marginal $\hat{\mu}$ is invariant with respect to $Q\left(\cdot\mbox{ | }\cdot,\,\phi\right)$,
i.e. $\hat{\mu}\left(B\right)=\int_{S}Q\left(B\mbox{ | }s,\phi\right)\hat{\mu}\left(ds\right)$
for all $B\in\mathcal{B}\left(S\right)$.
\end{defn}
When $\mu$ is stable, the long-run expected average cost $R\left(\phi,\hat{\mu}\right)$
is
\[
R\left(\phi,\hat{\mu}\right)=\liminf_{T\rightarrow\infty}\frac{1}{T}\mathbb{E}_{\hat{\mu}}^{\phi}\left[\sum_{t=0}^{T-1}r\left(s_{t},a_{t}\right)\right]=\langle\mu,r\rangle,
\]
by the individual ergodic theorem \cite[Page 388, Theorem 6]{Yosida80}.
Then for stable $\mu=\hat{\mu}\cdot\phi\in\mathcal{P}\left(K\right)$,
it follows that

\[
R\left(\phi,\hat{\mu}\right)=\langle\mu,r\rangle=\int_{S}r\left(s,\phi\right)\hat{\mu}\left(ds\right).
\]
Similarly, for stable $\mu=\hat{\mu}\cdot\phi$, it is true that

\[
Z_{\eta}\left(\phi,\hat{\mu}\right)=\langle\mu,\left(z-\eta\right)_{-}\rangle=\int_{S}\left[\int_{A}\left(z\left(s,a\right)-\eta\right)_{-}\phi\left(da\mbox{ | }s\right)\right]\hat{\mu}\left(ds\right)
\]
for all $\eta\in\left[a,b\right]$.

To see the connection between problem (\ref{DOMINANCE}) - (\ref{DOMINANCE-2})
and stable policies, let $I_{\Gamma}$ be the indicator function of
a set $\Gamma$ in $\mathcal{B}\left(K\right)$. Define the \textit{occupation measure}
$\mu$ on $K$ via
\[
\mu_{\nu,T}^{\pi}\left(\Gamma\right)=\frac{1}{T}\sum_{t=0}^{T-1}\mathbb{E}_{\nu}^{\pi}\left\{ I_{\Gamma}\left(s_{t},a_{t}\right)\right\} =\frac{1}{T}\sum_{t=0}^{T-1}P_{\nu}^{\pi}\left\{ \left(s_{t},a_{t}\right)\in\Gamma\right\} 
\]
for all $\Gamma\in\mathcal{B}\left(K\right)$. Then,
\[
R\left(\pi,\nu\right)=\liminf_{T\rightarrow\infty}\frac{1}{T}\mathbb{E}_{\nu}^{\pi}\left[\sum_{t=0}^{T-1}r\left(s_{t},a_{t}\right)\right]=\liminf_{T\rightarrow\infty}\langle\mu_{\nu,T}^{\pi},r\rangle
\]
and
\[
Z_{\eta}\left(\phi,\hat{\mu}\right)=\liminf_{T\rightarrow\infty}\frac{1}{T}\mathbb{E}_{\nu}^{\pi}\left[\sum_{t=0}^{T-1}\left(z\left(s_{t},a_{t}\right)-\eta\right)_{-}\right]=\liminf_{T\rightarrow\infty}\langle\mu_{\nu,T}^{\pi},\left(z-\eta\right)_{-}\rangle
\]
for all $\eta\in\left[a,b\right]$.

To continue, we introduce some technical assumptions for the rest
of the paper. Let $\mathcal{C}_{b}\left(K\right)$ be the space of
continuous and bounded functions on $K$. The transition law $Q$
is defined to be \textit{weakly continuous} when $\int_{S}h\left(\xi\right)Q\left(d\xi\mbox{ | }\cdot\right)$
is in $\mathcal{C}_{b}\left(K\right)$ for all $h\in\mathcal{C}_{b}\left(K\right)$.
\begin{assumption}
\label{ass:reduction}

(a) Problem (\ref{DOMINANCE}) - (\ref{DOMINANCE-2}) is consistent,
i.e. the set $\Delta$ is nonempty.

(b) The reward function $r$ is nonpositive, and for any $\epsilon\geq0$
the set $\left\{ \left(s,a\right)\in S\times A\mbox{ : }r\left(s,a\right)\geq-\epsilon\right\} $
is compact. 

(c) The function $z\left(s,a\right)$ is bounded and upper semi-continuous
on $S\times A$.

(d) The transition law $Q$ is weakly continuous.
\end{assumption}
A function $f$ on $K$ is called a \textit{moment} if there exists
a nondecreasing sequence of compact sets $K_{n}\uparrow K$ such that
\[
\lim_{n\rightarrow\infty}\inf_{\left(s,a\right)\notin K_{n}}f\left(s,a\right)=\infty,
\]
see \cite[Definition E.7]{HL96}. When $K$ is compact, then any function
on $K$ is a moment. Assumption \ref{ass:reduction}(b) implies that
$-r$ is a moment. By construction, all of the functions $\left(z\left(s,a\right)-\eta\right)_{-}$
are bounded above by zero on $S\times A$ for all $\eta\in\left[a,b\right]$.

The next lemma reduces the search for optimal policies to stable policies.
We define

\[
\Delta_{s}\triangleq\left\{ \mu\in\mathcal{P}\left(K\right)\mbox{ : }\mu\mbox{ is stable},\,\mu=\hat{\mu}\cdot\phi\mbox{ and }\left(\phi,\hat{\mu}\right)\in\Delta\right\} 
\]
to be the set of all stable probability measures $\mu$ that are feasible
for problem (\ref{DOMINANCE}) - (\ref{DOMINANCE-2}).
\begin{lem}
\label{lem:reduction} Suppose assumption \ref{ass:reduction} holds.
For each feasible pair $\left(\pi,\nu\right)\in\Delta$, there exists
a stable probability measure $\mu=\hat{\mu}\cdot\phi$ such that $\left(\phi,\hat{\mu}\right)\in\Delta$
and $R\left(\pi,\nu\right)\leq R\left(\phi,\hat{\mu}\right)=\langle\mu,r\rangle$.\end{lem}
\begin{proof}
For any $\left(\pi,\nu\right)\in\Delta$, there exists a stable policy
$\mu=\hat{\mu}\cdot\phi$ such that
\[
R\left(\pi,\nu\right)\leq R\left(\phi,\hat{\mu}\right)=\langle\mu,r\rangle
\]
by \cite[Lemma 5.7.10]{HL96}. By the same reasoning,
\[
\mathbb{E}\left[\left(Y-\eta\right)_{-}\right]\leq Z_{\eta}\left(\pi,\nu\right)\leq Z_{\eta}\left(\phi,\hat{\mu}\right)=\langle\mu,\left(z-\eta\right)_{-}\rangle
\]
for all $\eta\in\left[a,b\right]$ so that $\mu=\hat{\mu}\cdot\phi$
is feasible.
\end{proof}
Problem (\ref{DOMINANCE}) - (\ref{DOMINANCE-2}) is \textit{solvable}
if there exists a pair $\left(\pi^{*},\nu^{*}\right)\in\Delta$ with
$R\left(\pi^{*},\nu^{*}\right)=\rho^{*}$, i.e. the optimal value
is attained. When an optimization problem is solvable, we can replace
`sup' and `inf' with `max' and `min'. We use the preceding lemma to
show that problem (\ref{DOMINANCE}) - (\ref{DOMINANCE-2}) is solvable.
\begin{thm}
Problem (\ref{DOMINANCE}) - (\ref{DOMINANCE-2}) is solvable.\end{thm}
\begin{proof}
By lemma \ref{lem:reduction}, 
\[
\rho^{*}=\sup\left\{ \langle\mu,r\rangle\mbox{ : }\mu\in\Delta_{s}\right\} .
\]
Now apply the proof of \cite[Theorem 5.7.9]{HL96}. Let $\left\{ \epsilon_{n}\right\} $
be a sequence with $\epsilon_{n}\downarrow0$ and $\epsilon_{n}\leq1$.
For any $\epsilon_{n}$, there is a pair $\left(\pi^{n},\nu^{n}\right)\in\Delta$
with $R\left(\pi^{n},\nu^{n}\right)\geq\rho^{*}-\epsilon_{n}$ by
the definition of $\rho^{*}$. Again, by lemma \ref{lem:reduction},
for each $\left(\pi^{n},\nu^{n}\right)\in\Delta$ there is a pair
$\left(\phi^{n},\hat{\mu}^{n}\right)\in\Delta$ such that $\mu^{n}=\hat{\mu}^{n}\cdot\phi^{n}$
is stable and $R\left(\pi^{n},\nu^{n}\right)\leq R\left(\phi^{n},\hat{\mu}^{n}\right)=\langle\mu^{n},r\rangle$.

By construction, $\langle\mu^{n},r\rangle\geq\rho^{*}-\epsilon_{n}$
and $\epsilon_{n}\in\left(0,1\right)$ for all $n$, so $\inf_{n}\langle\mu^{n},r\rangle\geq\rho^{*}-1$.
It follows that $\sup_{n}\langle\mu^{n},-r\rangle\leq1-\rho^{*}$.
Since $-r$ is a moment, the preceding inequality along with \cite[Proposition E.8]{HL96}
and \cite[Proposition E.6]{HL96} imply that there exists a subsequence
of measures $\left\{ \mu^{n_{i}}\right\} $ converging weakly to a
measure $\mu$ on $K$. Now
\[
\rho^{*}\leq\limsup_{i\rightarrow\infty}\,\langle\mu^{n_{i}},r\rangle
\]
holds since $\langle\mu^{n},r\rangle\geq\rho^{*}-\epsilon_{n}$ for
all $n$ and $\epsilon_{n}\downarrow0$. By \cite[Proposition E.2]{HL96},
\[
\limsup_{i\rightarrow\infty}\,\langle\mu^{n_{i}},r\rangle\leq\langle\mu,r\rangle,
\]
so we obtain
\[
\rho^{*}\leq\langle\mu,r\rangle.
\]
Since $\langle\mu,r\rangle\leq\rho^{*}$ must hold by definition of
$\rho^{*}$, the preceding inequality shows that $\langle\mu,r\rangle=\rho^{*}$,
i.e. $\mu$ attains the optimal value $\rho^{*}$ and is stable. By
a similar argument,
\[
\mathbb{E}\left[\left(Y-\eta\right)_{-}\right]\leq\limsup_{i\rightarrow\infty}\,\langle\mu^{n_{i}},\left(z-\eta\right)_{-}\rangle\leq\langle\mu,\left(z-\eta\right)_{-}\rangle
\]
since each $\langle\mu^{n_{i}},\left(z-\eta\right)_{-}\rangle\geq\mathbb{E}\left[\left(Y-\eta\right)_{-}\right]$
for all $i$ and all $\eta\in\left[a,b\right]$. Thus, $\mu$ is feasible.

Let $\mu^{*}$ be the optimal stable measure just guaranteed, and
disintegrate to obtain $\mu^{*}=\hat{\mu}^{*}\cdot\phi^{*}$. The
pair $\left(\phi^{*},\,\hat{\mu}^{*}\right)$ is then optimal for
problem (\ref{DOMINANCE}) - (\ref{DOMINANCE-2}) since
\[
R\left(\phi^{*},\,\hat{\mu}^{*}\right)=\langle\mu^{*},r\rangle=\rho^{*},
\]
and
\[
Z_{\eta}\left(\phi^{*},\,\hat{\mu}^{*}\right)=\langle\mu^{*},\left(z-\eta\right)_{-}\rangle\geq\mathbb{E}\left[\left(Y-\eta\right)_{-}\right]
\]
for all $\eta\in\left[a,b\right]$.
\end{proof}
From the preceding theorem, we can now write maximization instead
of supremum in the objective of problem (\ref{DOMINANCE}) - (\ref{DOMINANCE-2}),

\[
\rho^{*}\triangleq\max\left\{ R\left(\pi,\nu\right)\mbox{ : }\left(\pi,\nu\right)\in\Delta_{s}\right\} .
\]

We are now ready to formalize problem (\ref{DOMINANCE}) - (\ref{DOMINANCE-2})
as a linear program. Introduce the weight function
\[
w\left(s,a\right)=1-r\left(s,a\right)
\]
on $K$. Under our assumption that $r$ is nonpositive, $w$ is bounded
from below by one. The space of signed Borel measures on $K$ is denoted
$\mathcal{M}\left(K\right)$. With the preceding weight function,
define $\mathcal{M}_{w}\left(K\right)$ to be the space of signed
measures $\mu$ on $K$ such that
\[
\|\mu\|_{\mathcal{M}_{w}\left(K\right)}\triangleq\int_{K}w\left(s,a\right)|\mu|\left(d\left(s,a\right)\right)<\infty.
\]
We can identify elements in $\mathcal{M}_{w}\left(K\right)$ with
stable policies, and vice versa. First, observe that the space $\mathcal{M}_{w}\left(K\right)$
is contained in the set of stable probability measures. If $\|\mu\|_{\mathcal{M}_{w}\left(K\right)}<\infty$,
then certainly
\[
\langle\mu,r\rangle=\int_{K}r\left(s,a\right)\mu\left(d\left(s,a\right)\right)>-\infty
\]
since $1-r=w$. Conversely, if $\mu$ is a stable probability measure,
then it is an element of $\mathcal{M}_{w}\left(K\right)$ since
\[
\int_{K}w\left(s,a\right)|\mu|\left(d\left(s,a\right)\right)=\int_{K}(1-r\left(s,a\right))\mu\left(d\left(s,a\right)\right)=\mu\left(K\right)-\langle\mu,r\rangle<\infty.
\]
Also define the weight function
\[
\hat{w}\left(s\right)=1-\sup_{a\in A\left(s\right)}r\left(s,a\right)
\]
on $S$ which is also bounded from below by one. The space $\mathcal{M}_{\hat{w}}\left(S\right)$
is defined analogously with $\hat{w}$ and $S$ in place of $w$ and
$S\times A$.

The topological dual of $\mathcal{M}_{w}\left(K\right)$ is $\mathcal{F}_{w}\left(K\right)$,
the vector space of measurable functions $h\mbox{ : }K\rightarrow\mathbb{R}$
such that
\[
\|h\|_{\mathcal{F}_{w}\left(K\right)}\triangleq\sup_{\left(s,a\right)\in K}\frac{|h\left(s,a\right)|}{w\left(s,a\right)}<\infty.
\]
Certainly, $r\in\mathcal{F}_{w}\left(K\right)$ by definition of $w$
since
\[
\|r\|_{\mathcal{F}_{w}\left(K\right)}=\sup_{\left(s,a\right)\in K}\frac{|r\left(s,a\right)|}{w\left(s,a\right)}=\sup_{\left(s,a\right)\in K}\frac{|r\left(s,a\right)|}{1+|r\left(s,a\right)|}\leq1.
\]
Every element $h\in\mathcal{F}_{w}\left(K\right)$ induces a continuous
linear functional on $\mathcal{M}_{w}\left(K\right)$ defined by 
\[
\langle\mu,h\rangle\triangleq\int_{K}h\left(s,a\right)\mu\left(d\left(\left(s,a\right)\right)\right).
\]
The two spaces $\left(\mathcal{M}_{w}\left(K\right),\,\mathcal{F}_{w}\left(K\right)\right)$
are called a dual pair, and the duality pairing is the bilinear form
$\langle u,h\rangle\mbox{ : }\mathcal{M}_{w}\left(K\right)\times\mathcal{F}_{w}\left(K\right)\rightarrow\mathbb{R}$
just defined. The topological dual of $\mathcal{M}_{\hat{w}}\left(S\right)$
is $\mathcal{F}_{\hat{w}}\left(S\right)$, which is defined analogously
with $S$ and $\hat{w}$ in place of $K$ and $w$.

We can now make some additional technical assumptions.
\begin{assumption}
\label{ass:bounded}

(a) The function $\left(z-\eta\right)_{-}$ is an element of $\mathcal{F}_{w}\left(K\right)$
for all $\eta\in\left[a,b\right]$.

(b) The function $\int_{S}\hat{w}\left(\xi\right)Q\left(d\xi\,\vert\, s,\, a\right)\mbox{ : }S\times A\rightarrow\mathbb{R}$
is an element of $\mathcal{F}_{w}\left(K\right)$.
\end{assumption}
Notice that assumption \ref{ass:bounded}(a) is satisfied if $z\in\mathcal{F}_{w}\left(K\right)$.
To see this fact, reason that
\[
\|\left(z-\eta\right)_{-}\|_{\mathcal{F}_{w}\left(K\right)}\leq\|z-\eta\|_{\mathcal{F}_{w}\left(K\right)}\leq\|z\|_{\mathcal{F}_{w}\left(K\right)}+\|\eta\|_{\mathcal{F}_{w}\left(K\right)},
\]
where the first inequality follows from $|\left(z-\eta\right)_{-}|\leq|z-\eta|$.
The constant function $f\left(x\right)=\eta$ on $K$ is in $\mathcal{F}_{w}\left(K\right)$
since
\[
\|\eta\|_{\mathcal{F}_{w}\left(K\right)}=\sup_{\left(s,a\right)\in K}\frac{|\eta|}{w\left(s,a\right)}\leq|\eta|.
\]

The linear mapping $L{}_{0}\mbox{ : }\mathcal{M}_{w}\left(K\right)\rightarrow\mathcal{M}_{\hat{w}}\left(S\right)$
defined by
\begin{align}
\left[L_{0}\mu\right]\left(B\right)\triangleq\hat{\mu}\left(B\right)-\int_{K}Q\left(B\mbox{ | }s,a\right)\mu\left(d\left(s,a\right)\right),\hspace{0.2in} & \forall B\in\mathcal{B}\left(S\right),\label{eq:L0}
\end{align}
is used to verify that $\mu$ is an invariant probability measure
on $K$ with respect to $Q$. The mapping (\ref{eq:L0}) appears in
all work on convex analytic methods for long-run average reward/cost
MDPs. When $L_{0}\mu\left(B\right)=0$, it means that the long-run
proportion of time in state $B$ is equal to the rate at which the
system transitions to state $B$ from all state-action pairs $\left(s,a\right)\in K$.
\begin{lem}
The condition $\mu\in\Delta_{s}$ is equivalent to $\langle\mu,r\rangle>-\infty$
and
\begin{align*}
L_{0}\mu=0,\\
\langle\mu,1\rangle=1,\\
\langle\mu,\left(z-\eta\right)_{-}\rangle\geq\mathbb{E}\left[\left(Y-\eta\right)_{-}\right], & \forall\eta\in\left[a,b\right],\\
\mu\geq0.
\end{align*}
\end{lem}
\begin{proof}
The linear constraints $\langle\mu,1\rangle=\int_{K}\mu\left(d\left(s,a\right)\right)=1$
and $\mu\geq0$ just ensure that $\mu$ is a probability measure on
$K$. The condition $L_{0}\mu=0$ is equivalent to invariance of $\mu$
with respect to $Q$. For stable $\mu=\hat{\mu}\cdot\phi$, $R\left(\phi,\hat{\mu}\right)=\langle\mu,r\rangle>-\infty$
and $Z_{\eta}\left(\phi,\hat{\mu}\right)=\langle\mu,\left(z-\eta\right)_{-}\rangle$.
Since $Z_{\eta}\left(\phi,\hat{\mu}\right)\geq\mathbb{E}\left[\left(Y-\eta\right)_{-}\right]$
for all $\eta\in\left[a,b\right]$, the conclusion follows.
\end{proof}
Next we continue with the representation of the dominance constraints
(\ref{DOMINANCE-1}). We would like to express the constraints $\langle\mu,\left(z-\eta\right)_{-}\rangle\geq\mathbb{E}\left[\left(Y-\eta\right)_{-}\right]$
for all $\eta\in\left[a,b\right]$ through a single linear operator.
\begin{lem}
For any $\mu\in\mathcal{P}\left(K\right)$, $\langle\mu,\left(z-\eta\right)_{-}\rangle$
is uniformly continuous in $\eta$ on $\left[a,b\right]$.\end{lem}
\begin{proof}
Write $\langle\mu,\left(z-\eta\right)_{-}\rangle=\int_{K}\left(z\left(s,a\right)-\eta\right)_{-}\mu\left(d\left(s,a\right)\right)$.
Certainly, each function $\left(z\left(s,a\right)-\eta\right)_{-}$
is continuous in $\eta$ for fixed $s\times a$. Choose $\epsilon>0$
and $|\eta'-\eta|<\epsilon$. Then
\begin{align*}
 & |\left(z\left(s,a\right)-\eta'\right)_{-}-\left(z\left(s,a\right)-\eta\right)_{-}|\\
\leq & |z\left(s,a\right)-\eta'-z\left(s,a\right)+\eta|\\
\leq & \epsilon,
\end{align*}
by definition of $\left(x\right)_{-}$. It follows that

\begin{align*}
 & |\int_{S\times A}\left(z\left(s,a\right)-\eta'\right)_{-}\mu\left(d\left(s,a\right)\right)-\int_{K}\left(z\left(s,a\right)-\eta\right)_{-}\mu\left(d\left(s,a\right)\right)|\\
\leq & |\int_{K}\epsilon\,\mu\left(d\left(s,a\right)\right)|\\
= & \epsilon,
\end{align*}
since $\mu$ is a probability measure.
\end{proof}
The preceding lemma allows us to write the dominance constraints (\ref{DOMINANCE-1})
as a linear operator in the space of continuous functions. Recall
that we have assumed $\left[a,b\right]$ to be a compact set. Let
$\mathcal{C}\left(\left[a,b\right]\right)$ be the space of continuous
functions on $\left[a,b\right]$ in the supremum norm,
\[
\|f\|_{\mathcal{C}\left(\left[a,b\right]\right)}=\sup_{a\leq x\leq b}|f\left(x\right)|
\]
for $f\in\mathcal{C}\left(\left[a,b\right]\right)$. The topological
dual of $\mathcal{C}\left(\left[a,b\right]\right)$ is $\mathcal{M}\left(\left[a,b\right]\right)$,
the space of finite signed Borel measures on $\left[a,b\right]$.
Every measure $\Lambda\in\mathcal{M}\left(\left[a,b\right]\right)$
induces a continuous linear functional on $\mathcal{C}\left(\left[a,b\right]\right)$
through the bilinear form
\[
\langle\Lambda,f\rangle=\int_{a}^{b}f\left(\eta\right)\Lambda\left(d\eta\right).
\]

Define the linear operator $L_{1}\mbox{ : }\mathcal{M}_{w}\left(K\right)\rightarrow\mathcal{C}\left(\left[a,b\right]\right)$
by

\begin{align}
\left[L_{1}\mu\right]\left(\eta\right)\triangleq\langle\mu,\,\left(z-\eta\right)_{-}\rangle,\hspace{0.2in} & \forall\eta\in\left[a,b\right].\label{eq:L1}
\end{align}
Also define the continuous function $y\in\mathcal{C}\left(\left[a,b\right]\right)$
where $y\left(\eta\right)=\mathbb{E}\left[\left(Y-\eta\right)_{-}\right]$
is the shortfall in $Y$ at level $\eta$ for all $\eta\in\left[a,b\right]$.
The dominance constraints are then equivalent to $\left[L_{1}\mu\right]\left(\eta\right)\geq y\left(\eta\right)$
for all $\eta\in\left[a,b\right]$, which can be written as the single
inequality $L_{1}\mu\geq y$ in $\mathcal{C}\left(\left[a,b\right]\right)$.

The linear programming form of problem (\ref{DOMINANCE}) - (\ref{DOMINANCE-2})
is
\begin{align}
\max\hspace{0.2in} & \langle\mu,\, r\rangle\label{LP}\\
\mbox{s.t.}\hspace{0.2in} & L_{0}\mu=0,\label{LP-1}\\
 & \langle\mu,1\rangle=1,\label{LP-2}\\
 & L_{1}\mu\geq y,\label{LP-3}\\
 & \mu\in\mathcal{M}_{w}\left(K\right),\,\mu\geq0.\label{LP-4}
\end{align}
Since $\rho^{*}\triangleq\max\left\{ R\left(\pi,\nu\right)\mbox{ : }\left(\pi,\nu\right)\in\Delta_{s}\right\} $,
and stable probability measures on $K$ can be identified as elements
of $\mathcal{M}_{w}\left(K\right)$, problem (\ref{LP}) - (\ref{LP-4})
is equivalent to problem (\ref{DOMINANCE}) - (\ref{DOMINANCE-2}).

\section{Establishing strong duality}

In this section we apply infinite-dimensional linear programming duality
to obtain the strong dual to problem (\ref{LP}) - (\ref{LP-4}).
The development in \cite{AN87} is behind our duality development,
and the duality theory for linear programming for MDPs on Borel spaces
in general.

We will introduce Lagrange multipliers for constraints (\ref{LP-1}),
(\ref{LP-2}), and (\ref{LP-3}), each Lagrange multiplier is drawn
from the appropriate topological dual space. Introduce Lagrange multipliers
$h\in\mathcal{F}_{\hat{w}}\left(S\right)$ for constraint (\ref{LP-1}).
The constraint $\langle\mu,1\rangle=1$ is an equality in $\mathbb{R}$,
so introduce Lagrange multipliers $\beta\in\mathbb{R}$ for constraint
(\ref{LP-2}). Finally, introduce Lagrange multipliers $\Lambda\in\mathcal{M}\left(\left[a,b\right]\right)$
for constraints (\ref{LP-3}). The Lagrange multipliers $\left(h,\beta,\Lambda\right)\in\mathcal{F}_{w}\left(S\right)\times\mathbb{R}\times\mathcal{M}\left(\left[a,b\right]\right)$
will be the decision variables in the upcoming dual to problem (\ref{LP})
- (\ref{LP-4}).

To proceed with duality, we compute the adjoints of $L_{0}$ and $L_{1}$.
The adjoint is analogous to the transpose for linear operators in
Euclidean spaces.
\begin{lem}
(a) The adjoint of $L_{0}$ is $L_{0}^{*}\mbox{ : }\mathcal{F}_{\hat{w}}\left(S\right)\rightarrow\mathcal{F}_{w}\left(K\right)$
where

\[
\left[L_{0}^{*}h\right]\left(s,a\right)\triangleq h\left(s\right)-\int_{S}h\left(\xi\right)Q\left(d\xi\,\vert\, s,\, a\right)
\]
for all $\left(s,a\right)\in K$.

(b) The adjoint of $L_{1}$ is $L_{1}^{*}\mbox{ : }\mathcal{M}\left(\left[a,b\right]\right)\rightarrow\mathcal{F}_{w}\left(K\right)$
where
\[
\left[L_{1}^{*}\Lambda\right]\left(s,a\right)=\int_{a}^{b}\left(z\left(s,a\right)-\eta\right)_{-}\Lambda\left(d\left(s,a\right)\right).
\]
\end{lem}
\begin{proof}
(a) This result is well known, see \cite{HL96,HL99}.

(b) Write
\begin{alignat*}{1}
\langle\Lambda,L_{1}\mu\rangle= & \int_{a}^{b}\langle\mu,\,\left(z-\eta\right)_{-}\rangle\Lambda\left(d\eta\right)\\
 & \int_{a}^{b}\left[\int_{K}\left(z\left(s,a\right)-\eta\right)_{-}\rangle\mu\left(d\left(s,a\right)\right)\right]\Lambda\left(d\eta\right).
\end{alignat*}
When $z$ is bounded on $S\times A$, then
\begin{align*}
|\int_{K}\left(z\left(s,a\right)-\eta\right)_{-}\left(\mu\times\Lambda\right)\left(d\left(\left(s,a\right)\times\eta\right)\right)|= & |\int_{K}\frac{\left(z\left(s,a\right)-\eta\right)_{-}}{w\left(s,a\right)}w\left(s,a\right)\left(\mu\times\Lambda\right)\left(d\left(\left(s,a\right)\times\eta\right)\right)|\\
\leq & \|\left(z-\eta\right)_{-}\|_{\mathcal{F}_{w}\left(K\right)}\,\|\mu\|_{\mathcal{M}_{w}\left(K\right)}\|\Lambda\|_{\mathcal{M}\left(\left[a,b\right]\right)}\\
< & \infty,
\end{align*}
since $\|\mu\|_{\mathcal{M}\left(K\right)}=1$ and $\|\Lambda\|_{\mathcal{M}\left(\left[a,b\right]\right)}<\infty$.
The Fubini theorem applies to justify interchange of the order of
integration,
\begin{align*}
\langle\Lambda,L_{1}x\rangle= & \int_{a}^{b}\left[\int_{K}\left(z\left(s,a\right)-\eta\right)_{-}\rangle\mu\left(d\left(s,a\right)\right)\right]\Lambda\left(d\eta\right)\\
= & \int_{K}\int_{a}^{b}\left(z\left(s,a\right)-\eta\right)_{-}\rangle\Lambda\left(d\eta\right)\mu\left(d\left(s,a\right)\right)\\
= & \int_{K}\langle\Lambda,\left(z\left(s,a\right)-\eta\right)_{-}\rangle\mu\left(d\left(s,a\right)\right),
\end{align*}
revealing $L_{1}^{*}\mbox{ : }\mathcal{M}\left(\left[a,b\right]\right)\rightarrow\mathcal{F}_{w}\left(K\right)$.
\end{proof}
We obtain the dual to problem (\ref{LP}) - (\ref{LP-4}) in the next
theorem.
\begin{thm}
The dual to problem (\ref{LP}) - (\ref{LP-4}) is
\begin{align}
\inf\hspace{0.2in} & \beta-\langle\Lambda,y\rangle\label{LP_dual}\\
\mbox{s.t.}\hspace{0.2in} & r+L_{0}^{*}h-\beta\,1+L_{1}^{*}\Lambda\leq0,\label{LP_dual-1}\\
 & \left(h,\beta,\Lambda\right)\in\mathcal{F}_{\hat{w}}\left(S\right)\times\mathbb{R}\times\mathcal{M}\left(\left[a,b\right]\right),\,\Lambda\geq0.\label{LP_dual-2}
\end{align}
\end{thm}
\begin{proof}
The Lagrangian for problem (\ref{LP}) - (\ref{LP-4}) is
\[
\vartheta\left(\mu,h,\beta,\Lambda\right)\triangleq\langle\mu,r\rangle+\langle h,L_{0}\mu\rangle+\beta\left(\langle\mu,1\rangle-1\right)+\langle\Lambda,L_{1}\mu-y\rangle,
\]
allowing problem (\ref{LP}) - (\ref{LP-4}) to be expressed as

\[
\max_{\mu\in\mathcal{M}_{w}\left(K\right)}\left\{ \inf_{\left(h,\beta,\Lambda\right)\in\mathcal{F}_{\hat{w}}\left(S\right)\times\mathbb{R}\times\mathcal{M}\left(\left[a,b\right]\right)}\left\{ \vartheta\left(\mu,h,\beta,\Lambda\right)\mbox{ : }\Lambda\geq0\right\} \mbox{ : }\mu\geq0\right\} .
\]
We rearrange the Lagrangian to obtain
\begin{align*}
\vartheta\left(\mu,h,\beta,\Lambda\right)= & \langle\mu,r\rangle+\langle h,L_{0}\mu\rangle+\beta\left(\langle\mu,1\rangle-1\right)+\langle\Lambda,L_{1}\mu-y\rangle\\
= & \langle\mu,r\rangle+\langle L_{0}^{*}h,\mu\rangle+\langle\mu,\,\beta\,1\rangle-\beta+\langle L_{1}^{*}\Lambda,\mu\rangle-\langle\Lambda,\, y\rangle\\
= & \langle\mu,r+L_{0}^{*}h+\beta\,1+L_{1}^{*}\Lambda\rangle-\beta-\langle\Lambda,y\rangle.
\end{align*}
The dual to problem (\ref{LP}) - (\ref{LP-4}) is then
\[
\inf_{\left(h,\beta,\Lambda\right)\in\mathcal{F}_{\hat{w}}\left(S\right)\times\mathbb{R}\times\mathcal{M}\left(\left[a,b\right]\right)}\left\{ \max_{\mu\in\mathcal{M}_{w}\left(K\right)}\left\{ \vartheta\left(\mu,h,\beta,\Lambda\right)\mbox{ : }\mu\geq0\right\} \mbox{ : }\Lambda\geq0\right\} .
\]
Since $\mu\geq0$, the constraint $r+L_{0}^{*}h+\beta\,1+L_{1}^{*}\Lambda\leq0$
is implied. Since $\beta$ is unrestricted, take $\beta=-\beta$ to
get the desired form.
\end{proof}
We write problem (\ref{LP_dual}) - (\ref{LP_dual-2}) with the infimum
objective rather than the minimization objective because we must verify
that the optimal value is attained. The dual problem (\ref{LP_dual})
- (\ref{LP_dual-2}) is explicitly
\begin{align}
\inf\hspace{0.2in} & \beta-\int_{a}^{b}\mathbb{E}\left[\left(Y-\eta\right)_{-}\right]\Lambda\left(d\eta\right)\label{LP_dual-3}\\
\mbox{s.t.}\hspace{0.2in} & r\left(s,a\right)+\int_{a}^{b}\left(z\left(s,a\right)-\eta\right)_{-}\Lambda\left(d\eta\right)\leq\beta+h\left(s\right)-\int_{S}h\left(\xi\right)Q\left(d\xi\mbox{ | }s,\, a\right), & \forall\left(s,a\right)\in K,\label{LP_dual-4}\\
 & \left(h,\beta,\Lambda\right)\in\mathcal{F}_{\hat{w}}\left(S\right)\times\mathbb{R}\times\mathcal{M}\left(\left[a,b\right]\right),\,\Lambda\geq0.\label{LP_dual-5}
\end{align}
Since $r\leq0,$ problem (\ref{LP_dual-3}) - (\ref{LP_dual-5}) is
readily seen to be consistent by choosing $h=0$, $\beta=0$, and
$\Lambda=0$.

Problem (\ref{LP_dual-3}) - (\ref{LP_dual-5}) has another, more
intuitive form. In \cite{Dentcheva2003,Dentcheva2004,Dentcheva2008},
it is recognized that the Lagrange multipliers of stochastic dominance
constraints are utility functions. This result is true in our case
as well. Using the family $\left\{ \left(x-\eta\right)_{-}\mbox{ : }\eta\in\left[a,b\right]\right\} $,
any measure $\Lambda\in\mathcal{M}\left(\left[a,b\right]\right)$
induces an increasing concave function in $\mathcal{C}\left(\left[a,b\right]\right)$
defined by
\[
u\left(x\right)=\int_{a}^{b}\left(x-\eta\right)_{-}\Lambda\left(d\eta\right)
\]
for all $x\in\mathbb{R}$. In fact, the above definition of $u$ gives
a function in $\mathcal{C}\left(\mathbb{R}\right)$ as well. Define

\begin{align*}
\mathcal{U}\left(\left[a,b\right]\right)= & \mbox{cl}\,\mbox{cone}\,\left\{ \left(x-\eta\right)_{-}\mbox{ : }\eta\in\left[a,b\right]\right\} \\
= & \left\{ u\left(x\right)=\int_{a}^{b}\left(x-\eta\right)_{-}\Lambda\left(d\eta\right)\mbox{ for }\Lambda\in\mathcal{M}\left(\left[a,b\right]\right),\,\Lambda\geq0\right\} 
\end{align*}
to be the closure of the cone generated by the family $\left\{ \left(x-\eta\right)_{-}\mbox{ : }\eta\in\left[a,b\right]\right\} $.
The set $\mathcal{U}\left(\left[a,b\right]\right)\subset\mathcal{U}\left(\mathbb{R}\right)$
is the set of all utility functions that can be constructed by limits
of sums of scalar multiplies of functions in $\left\{ \left(x-\eta\right)_{-}\mbox{ : }\eta\in\left[a,b\right]\right\} $.
\begin{cor}
\textup{Problem (\ref{LP_dual-3}) - (\ref{LP_dual-5}) is equivalent
to}
\begin{align}
\inf\hspace{0.2in} & \beta-\mathbb{E}\left[u\left(Y\right)\right]\label{LP_dual-6}\\
\mbox{s.t.}\hspace{0.2in} & r\left(s,a\right)+u\left(z\left(s,a\right)\right)\leq\beta+h\left(s\right)-\int_{S}h\left(\xi\right)Q\left(d\xi\,\vert\, s,\, a\right), & \forall\left(s,a\right)\in K,\label{LP_dual-7}\\
 & \left(h,\beta,u\right)\in\mathcal{F}_{\hat{w}}\left(S\right)\times\mathbb{R}\times\mathcal{U}\left(\left[a,b\right]\right).\label{LP_dual-8}
\end{align}
\end{cor}
\begin{proof}
Notice that the function
\[
u\left(x\right)=\int_{a}^{b}\left(x-\eta\right)_{-}\Lambda\left(d\eta\right)
\]
is an increasing concave function in $x$ for any $\Lambda\in\mathcal{M}\left(\left[a,b\right]\right)$
with $\Lambda\geq0$. By using this definition of $u$, we see that
for each state-action pair $\left(s,a\right)$,
\[
\langle\Lambda,\left(z\left(s,a\right)-\eta\right)_{-}\rangle=\int_{a}^{b}\left(z\left(s,a\right)-\eta\right)_{-}\Lambda\left(d\eta\right)=u\left(z\left(s,a\right)\right).
\]
Further, we can apply the Fubini theorem again to obtain
\[
\langle\Lambda,y\rangle=\int_{a}^{b}\mathbb{E}\left[\left(Y-\eta\right)_{-}\right]\Lambda\left(d\eta\right)=\mathbb{E}\left[\int_{a}^{b}\left(Y-\eta\right)_{-}\Lambda\left(d\eta\right)\right]=\mathbb{E}\left[u\left(Y\right)\right].
\]

\end{proof}
Next we verify that there is no duality gap between the primal problem
(\ref{LP}) - (\ref{LP-4}) and its dual (\ref{LP_dual}) - (\ref{LP_dual-2}).
All three dual problems (\ref{LP_dual}) - (\ref{LP_dual-2}), (\ref{LP_dual-3})
- (\ref{LP_dual-5}), and (\ref{LP_dual-6}) - (\ref{LP_dual-8})
are equivalent so the upcoming results apply to all of them.

The following result states that the optimal values of problems (\ref{LP})
- (\ref{LP-4}) and (\ref{LP_dual}) - (\ref{LP_dual-2}) are equal.
Afterwards, we will show that the optimal value of problem (\ref{LP_dual})
- (\ref{LP_dual-2}) is attained, establishing strong duality.
\begin{thm}
The optimal values of problems (\ref{LP}) - (\ref{LP-4}) and (\ref{LP_dual})
- (\ref{LP_dual-2}) are equal,
\begin{align*}
\rho^{*}= & \max\left\{ R\left(\pi,\nu\right)\mbox{ : }\left(\pi,\nu\right)\in\Delta\right\} \\
= & \inf\left\{ \beta-\langle\Lambda,y\rangle\mbox{ : }(\ref{LP_dual-1}),\,\left(h,\beta,\Lambda\right)\in\mathcal{F}_{\hat{w}}\left(S\right)\times\mathbb{R}\times\mathcal{M}\left(\left[a,b\right]\right),\,\Lambda\geq0\right\} .
\end{align*}
\end{thm}
\begin{proof}
Apply \cite[Theorem 12.3.4]{HL99}, which in turn follows from \cite[Theorem 3.9]{AN87}.
Introduce slack variables $\alpha\in\mathcal{C}\left(\left[a,b\right]\right)$
for the dominance constraints $L_{1}\mu\geq y$. We must show that
the set
\[
H\triangleq\left\{ \left(L_{0}\mu,\,\langle\mu,1\rangle,\, L_{1}x-\alpha,\,\langle\mu,r\rangle-\zeta\right)\mbox{ : }\mu\geq0,\,\alpha\geq0,\,\zeta\geq0\right\} 
\]
is weakly closed (closed in the weak topology). Let $\left(D,\,\leq\right)$
be a directed (partially ordered) set, and consider a net
\[
\left\{ \left(\mu_{\kappa},\alpha_{\kappa},\zeta_{\kappa}\right)\mbox{ : }\kappa\in D\right\} 
\]
where $\mu_{\kappa}\geq0$, $\alpha_{\kappa}\geq0$, and $\zeta_{\kappa}\geq0$
in $\mathcal{M}_{w}\left(K\right)\times\mathbb{R}\times\mathcal{C}\left(\left[a,b\right]\right)$
such that 
\[
\left(L_{0}\mu_{\kappa},\,\langle\mu_{\kappa},1\rangle,\, L_{1}\mu_{\kappa}-\alpha_{\kappa},\,\langle\mu_{\kappa},r\rangle-\zeta_{\kappa}\right)
\]
has weak limit $\left(\nu^{*},\gamma^{*},f^{*},\rho^{*}\right)\in\mathcal{M}_{\hat{w}}\left(S\right)\times\mathbb{R}\times\mathcal{C}\left(\left[a,b\right]\right)\times\mathbb{R}$.
Specifically,
\[
\langle\mu_{\kappa},1\rangle\rightarrow\gamma^{*}
\]
and
\[
\langle\mu_{\kappa},r\rangle-\zeta_{\kappa}\rightarrow\rho^{*},
\]
since weak convergence on $\mathbb{R}$ is equivalent to the usual
notion of convergence,
\[
\langle L_{0}\mu_{\kappa},g\rangle\rightarrow\langle\nu^{*},g\rangle
\]
for all $g\in\mathcal{F}_{\hat{w}}\left(S\right)$, and
\[
\langle L_{1}\mu_{\kappa}-\alpha_{\kappa},\Lambda\rangle\rightarrow\langle f^{*},\Lambda\rangle
\]
for all $\Lambda\in\mathcal{M}\left(\left[a,b\right]\right)$. We
must show that $\left(\nu^{*},\gamma^{*},f^{*},\rho^{*}\right)\in H$
under these conditions, i.e. that there exist $x\geq0$, $\alpha\geq0$,
and $\zeta\geq0$ such that
\[
\nu^{*}=L_{0}\mu,\,\gamma^{*}=\langle\mu,1\rangle,\, f^{*}=L_{1}\mu-\alpha,\,\rho^{*}=\langle\mu,r\rangle-\zeta.
\]
The fact that there exist $\mu\geq0$ and $\zeta\geq0$ such that
\[
\nu^{*}=L_{0}\mu,\,\gamma^{*}=\langle\mu,1\rangle,\,\rho^{*}=\langle\mu,r\rangle-\zeta,
\]
is already established in \cite[Theorem 12.3.4]{HL99}, and applies
to our setting without modification.

It remains to verify that there exists $\alpha\in\mathcal{C}\left(\left[a,b\right]\right)$
with $\alpha\geq0$ and $f^{*}=L_{1}\mu-\alpha$. Choose $\Lambda=\delta_{\eta}$
for the Dirac delta function at $\eta\in\left[a,b\right]$ to see
that
\[
\left[L_{1}\mu_{\kappa}\right]\left(\eta\right)-\alpha_{\kappa}\left(\eta\right)\rightarrow f^{*}\left(\eta\right)
\]
for all $\eta\in\left[a,b\right]$, establishing pointwise convergence.
Pointwise convergence on a compact set implies uniform convergence,
so in fact
\[
L_{1}\mu_{\kappa}-\alpha_{\kappa}\rightarrow f^{*}
\]
in the supremum norm topology on $\mathcal{C}\left(\left[a,b\right]\right)$.
Since $L_{1}\mu_{\kappa}\in\mathcal{C}\left(\left[a,b\right]\right)$
and $f^{*}\in\mathcal{C}\left(\left[a,b\right]\right)$, it follows
that $L_{1}\mu_{\kappa}-f^{*}\in\mathcal{C}\left(\left[a,b\right]\right)$
for any $\kappa$. Define $\alpha_{\kappa}=L_{1}\mu_{\kappa}-f^{*}$
and $\alpha=L_{1}\mu-f^{*}$, and notice that $\alpha\geq0$ necessarily.
\end{proof}
The next theorem shows that the dual problem (\ref{LP_dual}) - (\ref{LP_dual-2})
is solvable, i.e. there exists $\left(h^{*},\beta^{*},\Lambda^{*}\right)$
satisfying $r+L_{0}^{*}h^{*}-\beta^{*}\,1+L_{1}^{*}\Lambda^{*}\leq0$
that attain the optimal value
\[
\beta^{*}-\langle\Lambda^{*},y\rangle=\rho^{*}.
\]
When problem (\ref{LP_dual}) - (\ref{LP_dual-2}) is solvable, we
are justified in saying that strong duality holds: the optimal values
of both problems (\ref{LP}) - (\ref{LP-4}) and (\ref{LP_dual})
- (\ref{LP_dual-2}) are equal and both problems attain their optimal
value.

To continue we make some assumptions in line with \cite{HGL03}.
\begin{assumption}
\label{ass:strong} There exists a minimizing sequence $\left(h^{n},\beta^{n},\Lambda^{n}\right)$
in problem (\ref{LP_dual}) - (\ref{LP_dual-2}) such that

(a) $\left\{ \beta^{n}\right\} $ is bounded in $\mathbb{R}$,

(b) $\left\{ h^{n}\right\} $ is bounded in $\mathcal{F}_{\hat{w}}\left(S\right)$,
and

(c) $\left\{ \Lambda^{n}\right\} $ is bounded in the weak{*} topology
on $\mathcal{M}\left(\left[a,b\right]\right)$.
\end{assumption}
We establish strong duality next. To reiterate, strong duality holds
when the optimal values of problems (\ref{LP}) - (\ref{LP-4}) and
(\ref{LP_dual}) - (\ref{LP_dual-2}) are equal, and both problems
are solvable.
\begin{thm}
Suppose assumption \ref{ass:strong} holds. Strong duality holds between
problem (\ref{LP}) - (\ref{LP-4}) and problem (\ref{LP_dual}) -
(\ref{LP_dual-2}).\end{thm}
\begin{proof}
Let $\left(h^{n},\beta^{n},\Lambda^{n}\right)\in\mathcal{F}_{\hat{w}}\left(S\right)\times\mathbb{R}\times\mathcal{M}\left(\left[a,b\right]\right)$
for $n\geq0$ be a minimizing sequence of triples given in the preceding
assumption \ref{ass:strong}:

\begin{align*}
r\left(s,a\right)+\int_{a}^{b}\left(z\left(s,a\right)-\eta\right)_{-}\Lambda^{n}\left(d\eta\right)\leq\beta^{n}+h^{n}\left(s\right)-\int_{S}h^{n}\left(\xi\right)Q\left(d\xi\mbox{ | }s,\, a\right),\hspace{0.2in} & \forall\left(s,a\right)\in K,
\end{align*}
for all $n\geq0$ and

\[
\beta^{n}-\int_{a}^{b}\mathbb{E}\left[\left(Y-\eta\right)_{-}\right]\Lambda^{n}\left(d\eta\right)\downarrow\rho^{*}.
\]
Since the sequence $\left\{ \beta^{n}\right\} $ is bounded, it has
a convergent subsequence with $\lim_{n\rightarrow\infty}\beta^{n}=\beta^{*}$.

Now $\left\{ \Lambda^{n}\right\} $ is bounded in $\mathcal{M}\left(\left[a,b\right]\right)$
in the weak{*} topology induced by $\mathcal{C}\left(\left[a,b\right]\right)$
by assumption. Since $\left\{ \Lambda^{n}\right\} $ is bounded, the
sequence can be scaled to lie in the closed unit ball of $\mathcal{M}\left(\left[a,b\right]\right)$
in the weak{*} topology. Since $\mathcal{C}\left(\left[a,b\right]\right)$
is separable (there exists a countable dense set, i.e. the polynomials
with rational coefficients), the weak{*} topology on $\mathcal{M}\left(\left[a,b\right]\right)$
is metrizable. By the Banach-Alaoglu theorem, it follows that $\left\{ \Lambda^{n}\right\} $
has a subsequence that converges to some $\Lambda^{*}$ in the weak{*}
topology, i.e.

\[
\langle\Lambda^{n},f\rangle\rightarrow\langle\Lambda^{*},f\rangle
\]
for all $f\in\mathcal{C}\left(\left[a,b\right]\right)$. In particular,
since $\mathbb{E}\left[\left(Y-\eta\right)_{-}\right]$ and $\left(z\left(s,a\right)-\eta\right)_{-}$
are continuous functions on $\left[a,b\right]$ for all $\left(s,a\right)\in K$,
it follows that

\[
\lim_{n\rightarrow\infty}\int_{a}^{b}\mathbb{E}\left[\left(Y-\eta\right)_{-}\right]\Lambda^{n}\left(d\eta\right)=\int_{a}^{b}\mathbb{E}\left[\left(Y-\eta\right)_{-}\right]\Lambda^{*}\left(d\eta\right)
\]
and
\[
\lim_{n\rightarrow\infty}\int_{a}^{b}\left(z\left(s,a\right)-\eta\right)_{-}\Lambda^{n}\left(d\eta\right)=\int_{a}^{b}\left(z\left(s,a\right)-\eta\right)_{-}\Lambda^{*}\left(d\eta\right).
\]

Finally, since $\left\{ h^{n}\right\} $ is bounded in $\mathcal{F}_{\hat{w}}\left(S\right)$
we can define 
\[
h^{*}\left(s\right)\triangleq\liminf_{m\rightarrow\infty}h^{n}\left(s\right)
\]
for all $s\in S$. Then the function $h^{*}\left(s\right)$ is bounded
in $\mathcal{F}_{\hat{w}}\left(S\right)$, and
\[
\liminf_{n\rightarrow\infty}\int_{S}h^{n}\left(\xi\right)Q\left(d\xi\mbox{ | }s,a\right)\geq\int_{S}h^{*}\left(\xi\right)Q\left(d\xi\mbox{ | }s,a\right)
\]
by Fatou's lemma. Taking the limit, it follows that $\left(h^{*},\beta^{*},\Lambda^{*}\right)$
is an optimal solution to the dual problem.
\end{proof}
The role of the utility function $u$ in problem (\ref{LP_dual-6})
- (\ref{LP_dual-8}) is fairly intuitive. The function $u$ serves
as an additional pricing variable for the performance function $z\left(s,a\right)$,
and the total reward is treated as if it were $r\left(s,a\right)+u\left(z\left(s,a\right)\right)$.
Problem (\ref{LP_dual-6}) - (\ref{LP_dual-8}) leads to a new version
of the optimality equations for average reward based on infinite-dimensional
linear programming complementary slackness.
\begin{thm}
Let $\mu^{*}=\hat{\mu}^{*}\cdot\phi^{*}$ be an optimal solution to
problem (\ref{LP}) - (\ref{LP-4}), and $\left(h^{*},\beta^{*},u^{*}\right)$
be an optimal solution to problem (\ref{LP_dual}) - (\ref{LP_dual-2}).
Then
\[
\langle\mu^{*},u^{*}\left(z\right)\rangle=\mathbb{E}\left[u^{*}\left(Y\right)\right],
\]
and
\[
\beta^{*}+h^{*}\left(s\right)=\sup_{a\in A\left(s\right)}\left\{ r\left(s,a\right)+u^{*}\left(z\left(s,a\right)\right)+\int_{S}h^{*}\left(\xi\right)Q\left(d\xi\,\vert\, s,\, a\right)\right\} 
\]
for $\hat{\mu}^{*}-$almost all $s\in S$.\end{thm}
\begin{proof}
There is a corresponding optimal solution $\left(h^{*},\beta^{*},\Lambda^{*}\right)$
to problem (\ref{LP_dual}) - (\ref{LP_dual-2}). Complementary slackness
between problems (\ref{LP}) - (\ref{LP-4}) and (\ref{LP_dual})
- (\ref{LP_dual-2}) gives $\langle\Lambda^{*},\, L_{1}\mu^{*}-y\rangle=0$,
where $\left(h^{*},\beta^{*},u^{*}\right)$ is a corresponding optimal
solution of problem (\ref{LP_dual}) - (\ref{LP_dual-2}). Then
\[
\langle\Lambda^{*},\, L_{1}\mu^{*}\rangle=\langle L_{1}^{*}\Lambda^{*},\mu^{*}\rangle=\langle\mu^{*},u^{*}\left(z\right)\rangle
\]
and $\langle\Lambda^{*},\, y\rangle=\mathbb{E}\left[u^{*}\left(Y\right)\right]$.

Complementary slackness also gives 
\[
\langle r+L_{0}^{*}h^{*}-\beta^{*}\,1+L_{1}^{*}\Lambda^{*},\,\mu^{*}\rangle=0,
\]
which yields the second statement since $\mu^{*}\geq0$ and $r+L_{0}^{*}h^{*}-\beta^{*}\,1+L_{1}^{*}\Lambda^{*}\leq0$.
\end{proof}

\section{Variations and extensions}

\subsection{Multivariate integral stochastic orders}

We extend our repertoire in this section to include some additional
stochastic orders. Integral stochastic orders (see \cite{Muller2002})
refer to stochastic orders that are defined in terms of families of
functions. The increasing concave stochastic order is an example of
an integral stochastic order, because it is defined in terms of the
family of increasing concave functions. We now give attention to some
multivariate integral stochastic orders. So far, we have considered
a $z\mbox{ : }K\rightarrow\mathbb{R}$ that is a scalar-valued function.
In practice there are usually many system performance measures of
interest, so it is logical to consider vector valued $z\mbox{ : }K\rightarrow\mathbb{R}^{n}$
as well. For example, $z\left(s,a\right)$ may represent the service
rate to $n$ customers in a wireless network. The empirical distribution
$\lim_{T\rightarrow\infty}\frac{1}{T}\sum_{t=0}^{T-1}z\left(s_{t},a_{t}\right)$
is now a vector-valued random variable on $\mathbb{R}^{n}$.

Recall the multivariate increasing concave stochastic order. For random
vectors $X,\, Y\in\mathbb{R}^{n}$, $X$ dominates $Y$ in the increasing
concave stochastic order, written $X\geq_{icv}Y$, if $\mathbb{E}\left[u\left(X\right)\right]\geq\mathbb{E}\left[u\left(Y\right)\right]$
for all increasing concave functions $u\mbox{ : }\mathbb{R}^{n}\rightarrow\mathbb{R}$
such that both expectations exist. Unlike univariate $\geq_{icv}$,
there is no parametrized family of functions (like $\left(x-\eta\right)_{-}$)
that generates all the multivariate increasing concave functions.
This result rests on the fact that the set of extreme points of the
increasing concave functions on $\mathbb{R}^{n}$ to $\mathbb{R}$
is dense for $n\geq2$, see \cite{Johansen1974,Bronshtein1978}.

As in \cite{Dentcheva2009}, we can relax the condition $X\geq_{icv}Y$
by constructing a tractable parametrized family of increasing concave
functions. Let $u\left(\cdot;\,\xi\right)\mbox{ : }\mathbb{R}^{n}\rightarrow\mathbb{R}$
represent a family of increasing concave functions parametrized by
$\xi\in\Xi\subset\mathbb{R}^{p}$ where $\Xi$ is compact. Then, the
family of functions $\left\{ u\left(\cdot;\,\xi\right)\right\} _{\xi\in\Xi}$
is a subset of all increasing concave functions and leads to a relaxation
of $\geq_{icv}$. We say $X$ dominates $Y$ with respect to the integral
stochastic order generated by $\left\{ u\left(\cdot;\,\xi\right)\right\} _{\xi\in\Xi}$
if $\mathbb{E}\left[u\left(X;\,\xi\right)\right]\geq\mathbb{E}\left[u\left(Y;\,\xi\right)\right]$
for all $\xi\in\Xi$. Define
\[
Z_{\xi}\left(\pi,\nu\right)=\liminf_{T\rightarrow\infty}\frac{1}{T}\mathbb{E}_{\nu}^{\pi}\left[\sum_{t=0}^{T-1}u\left(z\left(s_{t},a_{t}\right);\,\xi\right)\right]
\]
for all $\xi\in\Xi$. For convenience, we assume $u\left(x;\,\xi\right)$
is continuous in $\xi\in\Xi$ for any $x\in\mathbb{R}^{n}$.

We propose the multivariate dominance-constrained MDP:

\begin{align}
\sup\hspace{0.2in} & R\left(\pi,\nu\right)\label{MULTI}\\
\mbox{s.t.}\hspace{0.2in} & Z_{\xi}\left(\pi,\nu\right)\geq\mathbb{E}\left[u\left(Y;\,\xi\right)\right], & \forall\xi\in\Xi,\label{MULTI-1}\\
 & \pi\in\Pi.\label{MULTI-2}
\end{align}
using $\left\{ u\left(\cdot;\,\xi\right)\right\} _{\xi\in\Xi}$.

By the same reasoning as earlier,

\[
Z_{\xi}\left(\phi,\hat{\mu}\right)=\langle\mu,\, u\left(z\left(s,a\right);\,\xi\right)\rangle=\int_{S}\left[\int_{A}u\left(z\left(s,a\right);\,\xi\right)\phi\left(da\mbox{ | }s\right)\right]\hat{\mu}\left(ds\right)
\]
for all $\xi\in\Xi$ when $\mu=\hat{\mu}\cdot\phi\in\Delta_{s}$.
\begin{lem}
For any $\mu\in\mathcal{P}\left(K\right)$, $\langle\mu,u\left(z;\,\xi\right)\rangle$
is continuous in $\xi$.\end{lem}
\begin{proof}
Write $\langle\mu,u\left(z;\,\xi\right)\rangle=\int_{K}u\left(z\left(s,a\right);\,\xi\right)\mu\left(d\left(s,a\right)\right)$.
Certainly each function $u\left(z\left(s,a\right);\,\xi\right)$ is
continuous in $\xi$ for any fixed $s\times a$. Since $\mu$ is finite,
it follows that the integral of $u\left(z\left(s,a\right);\,\xi\right)$
with respect to $\mu$ is continuous in $\xi$.
\end{proof}
Let $\mathcal{C}\left(\Xi\right)$ be the space of continuous functions
on $\Xi$ in the supremum norm,
\[
\|f\|_{\mathcal{C}\left(\Xi\right)}\triangleq\sup_{x\in\Xi}|f\left(\xi\right)|.
\]
We will express the dominance constraints (\ref{MULTI-1}) as a linear
operator in $\mathcal{C}\left(\Xi\right)$. This operator depends
on the parametrization $u\left(\cdot;\,\xi\right)$. The preceding
lemma justifies defining $L_{1}\mbox{ : }\mathcal{M}\left(S\times A\right)\rightarrow\mathcal{C}\left(\Xi\right)$
by

\begin{align}
\left[L_{1}x\right]\left(\xi\right)\triangleq\langle x,\, u\left(z;\,\xi\right)\rangle,\hspace{0.2in} & \xi\in\Xi.\label{eq:L1_multi}
\end{align}
Also define the continuous function $y\in\mathcal{C}\left(\Xi\right)$
by $y\left(\xi\right)=\mathbb{E}\left[u\left(Y;\,\xi\right)\right]$
for all $\xi\in\Xi$ to represent the benchmark.

The steady-state version of problem (\ref{MULTI}) - (\ref{MULTI-2})
is the modified linear program:

\begin{align}
\max\hspace{0.2in} & \langle\mu,\, r\rangle\label{LP_multi}\\
\mbox{s.t.}\hspace{0.2in} & L_{0}\mu=0,\label{LP_multi-1}\\
 & \langle\mu,1\rangle=1,\label{LP_multi-2}\\
 & L_{1}\mu\geq y,\label{LP_multi-3}\\
 & \mu\in\mathcal{M}_{w}\left(K\right),\,\mu\geq0.\label{LP_multi-4}
\end{align}
Problem (\ref{LP_multi}) - (\ref{LP_multi-4}) is almost the same
as problem (\ref{LP}) - (\ref{LP-4}), except that now $L_{1}\mu$
is an element in $\mathcal{C}\left(\Xi\right)$ to reflect the multivariate
dominance constraint.

We now compute the adjoint of $L_{1}$, which depends on the choice
of family $\left\{ u\left(\cdot;\,\xi\right)\mbox{ : }\xi\in\Xi\right\} $.
The parametrization $u\left(\cdot;\,\xi\right)$ will appear explicitly
in this computation.
\begin{lem}
The adjoint of $L_{1}$ is $L_{1}^{*}\mbox{ : }\mathcal{M}\left(\Xi\right)\rightarrow\mathcal{F}_{w}\left(K\right)$
where

\[
\left[L_{1}^{*}\Lambda\right]\left(s,a\right)\triangleq\int_{\Xi}u\left(z\left(s,a\right);\,\xi\right)\Lambda\left(d\xi\right).
\]
\end{lem}
\begin{proof}
Write
\begin{align*}
\langle\Lambda,L_{1}\mu\rangle= & \int_{\Xi}\langle\mu,\, u\left(z;\,\xi\right)\rangle\Lambda\left(d\xi\right)\\
= & \int_{\Xi}\left[\int_{K}u\left(z\left(s,a\right);\,\xi\right)\rangle\mu\left(d\left(s,a\right)\right)\right]\Lambda\left(d\xi\right).
\end{align*}
When $z$ is bounded on $S\times A$, then
\[
|\int u\left(z;\,\xi\right)\left(\mu\times\Lambda\right)\left(d\left(\left(s,a\right)\times\xi\right)\right)|\leq\|u\left(z\left(\cdot\right);\,\xi\right)\|_{\mathcal{F}_{w}\left(K\right)}\,\|\mu\|_{\mathcal{M}_{w}\left(K\right)}\|\Lambda\|_{\mathcal{M}\left(\left[a,b\right]\right)}<\infty.
\]
The Fubini theorem applies to justify interchange of the order of
integration,
\begin{align*}
\langle\Lambda,L_{1}\mu\rangle= & \int_{K}\left[\int_{\Xi}u\left(z\left(s,a\right);\,\xi\right)\Lambda\left(d\xi\right)\right]\mu\left(d\left(s,a\right)\right)\\
= & \int_{K}\langle\Lambda,u\left(z\left(s,a\right);\,\xi\right)\rangle\mu\left(d\left(s,a\right)\right).
\end{align*}

\end{proof}
The dual to problem (\ref{LP_multi}) - (\ref{LP_multi-4}) looks
identical to problem (\ref{LP_dual}) - (\ref{LP_dual-2}) and is
now explicitly
\begin{align}
\inf\hspace{0.2in} & \beta-\int_{\Xi}\mathbb{E}\left[u\left(Y;\,\xi\right)\right]\Lambda\left(d\xi\right)\label{LP_multi_dual}\\
\mbox{s.t.}\hspace{0.2in} & r\left(s,a\right)+\int_{\Xi}u\left(z\left(s,a\right);\,\xi\right)\Lambda\left(d\xi\right)\leq\beta+h\left(s\right)-\int_{S}h\left(\xi\right)Q\left(d\xi\mbox{ | }s,\, a\right), & \forall\left(s,a\right)\in K,\label{LP_multi_dual-1}\\
 & \left(h,\beta,\Lambda\right)\in\mathcal{F}_{\hat{w}}\left(S\right)\times\mathbb{R}\times\mathcal{M}\left(\Xi\right),\,\Lambda\geq0.\label{LP_multi_dual-2}
\end{align}
Define
\begin{align*}
\mathcal{U}\left(\Xi\right)= & \mbox{cl}\,\mbox{cone}\,\left\{ u\left(x;\,\xi\right)\mbox{ : }\xi\in\Xi\right\} \\
= & \left\{ u\left(x\right)=\int_{\Xi}u\left(x;\,\xi\right)\Lambda\left(d\xi\right)\mbox{ for }\Lambda\in\mathcal{M}\left(\Xi\right),\,\Lambda\geq0\right\} 
\end{align*}
to be the closure of the cone of functions generated by $\left\{ u\left(x;\,\xi\right)\mbox{ : }\xi\in\Xi\right\} $.
In this case $\mathcal{U}\left(\Xi\right)$ is a family of functions
in $\mathcal{C}\left(\mathbb{R}^{n}\right)$, the space of continuous
functions $f\mbox{ : }\mathbb{R}^{n}\rightarrow\mathbb{R}$. We see
immediately that problem (\ref{LP_multi_dual}) - (\ref{LP_multi_dual-2})
is equivalent to
\begin{align}
\inf\hspace{0.2in} & \beta-\mathbb{E}\left[u\left(Y\right)\right]\label{LP_multi_dual-3}\\
\mbox{s.t.}\hspace{0.2in} & r\left(s,a\right)+u\left(z\left(s,a\right)\right)\leq\beta+h\left(s\right)-\int_{S}h\left(\xi\right)Q\left(d\xi\mbox{ | }s,\, a\right), & \forall\left(s,a\right)\in K,\label{LP_multi_dual-4}\\
 & \left(h,\beta,u\right)\in\mathcal{F}_{\hat{w}}\left(S\right)\times\mathbb{R}\times\mathcal{U}\left(\Xi\right).\label{LP_multi_dual-5}
\end{align}
The variables $u\in\mathcal{U}\left(\Xi\right)$ in problem (\ref{LP_multi_dual-3})
- (\ref{LP_multi_dual-5}) are now pricing variables for the vector
$z$. When our earlier assumptions are suitably adapted, then strong
duality holds between problem (\ref{LP_multi}) - (\ref{LP_multi-4})
and problem (\ref{LP_multi_dual-3}) - (\ref{LP_multi_dual-5}).
\begin{thm}
The optimal values of problems (\ref{LP_multi}) - (\ref{LP_multi-4})
and (\ref{LP_multi_dual}) - (\ref{LP_multi_dual-2}) are equal. Further,
the dual problem (\ref{LP_multi_dual}) - (\ref{LP_multi_dual-2})
is solvable and strong duality holds between problems (\ref{LP_multi})
- (\ref{LP_multi-4}) and (\ref{LP_multi_dual}) - (\ref{LP_multi_dual-2}).
\end{thm}

\subsection{Discounted reward}

We briefly sketch the development for discounted reward, it is mostly
similar. Discounted cost MDPs in Borel spaces with finitely many constraints
are considered in \cite{HG00}. Introduce the discount factor $\delta\in\left(0,1\right)$
and consider the long-run expected discounted reward

\[
R\left(\pi,\nu\right)=\mathbb{E}_{\nu}^{\pi}\left[\sum_{t=0}^{\infty}\delta^{t}r\left(s_{t},a_{t}\right)\right].
\]
We are interested in the distribution of discounted reward $z$,
\[
\sum_{t=0}^{\infty}\delta^{t}z\left(s_{t},a_{t}\right).
\]

Define
\[
Z_{\eta}\left(\pi,\nu\right)\triangleq\mathbb{E}_{\nu}^{\pi}\left[\sum_{t=0}^{\infty}\delta^{t}\left(z\left(s_{t},a_{t}\right)-\eta\right)_{-}\right].
\]
We propose the dominance-constrained MDP:

\begin{align}
\sup\hspace{0.2in} & R\left(\pi,\nu\right)\label{DISCOUNT}\\
\mbox{s.t.}\hspace{0.2in} & Z_{\eta}\left(\pi,\nu\right)\geq\mathbb{E}\left[\left(Y-\eta\right)_{-}\right], & \forall\eta\in\left[a,b\right],\label{DISCOUNT-1}\\
 & \pi\in\Pi.\label{DISCOUNT-2}
\end{align}

We work with the $\delta-$discounted expected occupation measure
\[
\mu_{\nu}^{\pi}\left(\Gamma\right)\triangleq\sum_{t=0}^{\infty}\delta^{t}P_{\nu}^{\pi}\left(\left(s_{t},a_{t}\right)\in\Gamma\right)
\]
for all $\Gamma\in\mathcal{B}\left(S\times A\right)$. Now let
\begin{align}
\left[L_{0}\mu\right]\left(B\right)\triangleq\hat{\mu}\left(B\right)-\delta\int_{S\times A}Q\left(B\mbox{ | }s,a\right)\mu\left(d\left(s,a\right)\right),\hspace{0.2in} & \forall B\in\mathcal{B}\left(S\right),\label{eq:L0_discount}
\end{align}
and
\begin{align}
\left[L_{1}\mu\right]\left(\eta\right)\triangleq\langle\mu,\left(z-\eta\right)_{-}\rangle,\hspace{0.2in} & \forall\eta\in\left[a,b\right].\label{eq:L1_discount}
\end{align}
Also continue to define $y\in\mathcal{C}\left(\left[a,b\right]\right)$
by $y\left(\eta\right)=\mathbb{E}\left[\left(Y-\eta\right)_{-}\right]$
for all $\eta\in\left[a,b\right]$. Problem (\ref{DISCOUNT}) - (\ref{DISCOUNT-2})
is then equivalent to the linear program
\begin{align}
\max\hspace{0.2in} & \langle\mu,\, r\rangle\label{LP_DISCOUNT}\\
\mbox{s.t.}\hspace{0.2in} & L_{0}\mu=\nu,\label{LP_DISCOUNT-1}\\
 & L_{1}\mu\geq y,\label{LP_DISCOUNT-2}\\
 & \mu\in\mathcal{M}\left(K\right),\,\mu\geq0.\label{LP_DISCOUNT-3}
\end{align}

Introduce Lagrange multipliers $h\in\mathcal{F}_{\hat{w}}\left(S\right)$
for constraint $L_{0}\mu=\nu$ and multipliers $\Lambda\in\mathcal{M}\left(\left[a,b\right]\right)$
for constraint $L_{1}\mu\geq y$, the Lagrangian is then 
\[
\vartheta\left(\mu,h,\Lambda\right)=\langle\mu,r\rangle+\langle h,L_{0}\mu-\nu\rangle+\langle\Lambda,L_{1}\mu-y\rangle.
\]
The adjoint of $L_{0}$ is $L_{0}^{*}\mbox{ : }\mathcal{F}_{\hat{w}}\left(S\right)\rightarrow\mathcal{F}_{w}\left(S\times A\right)$
defined by
\[
\left[L_{0}^{*}h\right]\left(s,a\right)\triangleq h\left(s\right)-\delta\int_{S}h\left(\xi\right)Q\left(d\xi\mbox{ | }s,a\right).
\]
The adjoint of $L_{1}$ is still $L_{1}^{*}\mbox{ : }\mathcal{M}\left(\left[a,b\right]\right)\rightarrow\mathcal{F}_{w}\left(S\times A\right)$
where

\[
\left[L_{1}^{*}\Lambda\right]\left(s,a\right)\triangleq\int_{a}^{b}\left(z\left(s,a\right)-\eta\right)_{-}\Lambda\left(d\eta\right).
\]
The form of the dual follows.
\begin{thm}
The dual to problem (\ref{LP_DISCOUNT}) - (\ref{LP_DISCOUNT-3})
is
\begin{align}
\min\hspace{0.2in} & \langle h,\nu\rangle-\langle\Lambda,y\rangle\label{LP_DISCOUNT_dual}\\
\mbox{s.t.}\hspace{0.2in} & r+L_{0}^{*}h+L_{1}^{*}\Lambda\geq0,\label{LP_DISCOUNT_dual-1}\\
 & h\in\mathcal{F}_{w}\left(K\right),\,\Lambda\in\mathcal{M}\left(\left[a,b\right]\right),\,\Lambda\geq0.\label{LP_DISCOUNT_dual-2}
\end{align}
The optimal values of problems (\ref{LP_DISCOUNT}) - (\ref{LP_DISCOUNT-3})
and (\ref{LP_DISCOUNT_dual}) - (\ref{LP_DISCOUNT_dual-2}) are equal,
and problem (\ref{LP_DISCOUNT_dual}) - (\ref{LP_DISCOUNT_dual-2})
is solvable.
\end{thm}
This dual is explicitly
\begin{align}
\min\hspace{0.2in} & \langle h,\nu\rangle-\mathbb{E}\left[u\left(Y\right)\right]\label{LP_DISCOUNT_dual-3}\\
\mbox{s.t.}\hspace{0.2in} & r\left(s,a\right)+u\left(z\left(s,a\right)\right)\leq h\left(s\right)-\delta\int_{S}h\left(\xi\right)Q\left(d\xi\mbox{ | }s,\, a\right), & \forall\left(s,a\right)\in K,\label{LP_DISCOUNT_dual-4}\\
 & h\in\mathcal{F}_{w}\left(K\right),\, u\in\mathcal{U}\left(\left[a,b\right]\right).\label{LP_DISCOUNT_dual-5}
\end{align}
Problem (\ref{LP_DISCOUNT_dual-3}) - (\ref{LP_DISCOUNT_dual-5})
leads to a modified set of optimality equations for the infinite horizon
discounted reward case, namely
\[
h\left(s\right)=\max_{a\in A\left(s\right)}\left\{ r\left(s,a\right)+u\left(z\left(s,a\right)\right)+\delta\int_{S}h\left(\xi\right)Q\left(d\xi\mbox{ | }s,\, a\right)\right\} 
\]
for all $s\in S$.

\subsection{Approximate linear programming}

Various approaches have been put forward for solving infinite-dimensional
LPs with sequences of finite-dimensional LPs, such as in \cite{hernandez-lerma98,mendiondo08}.
Approximate linear programming (ALP) has been put forward as an approach
to the curse of dimensionality, and it can be applied to our present
setting. The average reward linear program (\ref{LP}) - (\ref{LP-4})
and the discounted reward linear program (\ref{LP_DISCOUNT}) - (\ref{LP_DISCOUNT-3})
generally have uncountably many variables and constraints.

ALP for average cost dynamic programming is developed in \cite{FR06}.
Previous work on ALP for dynamic programming has focused on approximating
the cost-to-go function $h$ rather than the steady-state occupation
measure $\mu$. It is more intuitive to design basis functions for
the cost-to-go function than the occupation measure. For problem (\ref{LP})
- (\ref{LP-4}), we approximate the cost-to-go function $h\in\mathcal{F}_{\hat{w}}\left(S\right)$
with the basis functions $\left\{ \phi_{1},\ldots,\phi_{m}\right\} \subset\mathcal{F}_{\hat{w}}\left(S\right)$.
We approximate the pricing variable $u\in\mathcal{U}\left(\left[a,b\right]\right)$
with basis functions $\left\{ u_{1},\ldots,u_{n}\right\} \subset\mathcal{U}\left(\left[a,b\right]\right)$.
The resulting approximate linear program is
\begin{align}
\min\hspace{0.2in} & \beta-\mathbb{E}\left[\sum_{i=1}^{n}\alpha_{i}u_{i}\left(Y\right)\right]\label{ALP}\\
\mbox{s.t.}\hspace{0.2in} & r\left(s,a\right)+\sum_{i=1}^{n}\alpha_{i}u_{i}\left(z\left(s,a\right)\right)\leq\beta+\sum_{j=1}^{m}\gamma_{j}h_{j}\left(s\right)-\int_{S}\left[\sum_{j=1}^{m}\gamma_{j}h_{j}\right]\left(\xi\right)Q\left(d\xi\,\vert\, s,\, a\right), & \forall\left(s,a\right)\in K,\label{ALP-1}\\
 & \left(\gamma,\beta,\alpha\right)\in\mathbb{R}^{m}\times\mathbb{R}\times\mathbb{R}^{n}.\label{ALP-2}
\end{align}
We are justified in writing minimization instead of infimum in problem
(\ref{ALP}) - (\ref{ALP-2}) because there are only finitely many
decision variables. ALP has been studied extensively for the linear
programming representation of the optimality equations for discounted
infinite horizon dynamic programming (see \cite{FR03,FR04,DFM09}).
The discounted approximate linear program is
\begin{align}
\min\hspace{0.2in} & \langle h,\nu\rangle-\mathbb{E}\left[\sum_{i=1}^{n}\alpha_{i}u_{i}\left(Y\right)\right]\label{ALP_DISCOUNT}\\
\mbox{s.t.}\hspace{0.2in} & r\left(s,a\right)+\sum_{i=1}^{n}\alpha_{i}u_{i}\left(z\left(s,a\right)\right)\leq\sum_{j=1}^{m}\gamma_{j}h_{j}\left(s\right)-\delta\int_{S}\left[\sum_{j=1}^{m}\gamma_{j}h_{j}\right]\left(\xi\right)Q\left(d\xi\mbox{ | }s,\, a\right), & \forall\left(s,a\right)\in K,\label{ALP_DISCOUNT-1}\\
 & \left(\gamma,\alpha\right)\in\mathbb{R}^{m}\times\mathbb{R}^{n}.\label{ALP_DISCOUNT-2}
\end{align}
Both problems (\ref{ALP}) - (\ref{ALP-2}) and (\ref{ALP_DISCOUNT})
- (\ref{ALP_DISCOUNT-2}) are restrictions of the corresponding problems
(\ref{LP_dual-6}) - (\ref{LP_dual-8}) and (\ref{LP_DISCOUNT_dual-3})
- (\ref{LP_DISCOUNT_dual-5}).

Problems (\ref{ALP}) - (\ref{ALP-2}) and (\ref{ALP_DISCOUNT}) -
(\ref{ALP_DISCOUNT-2}) have a manageable number of decision variables
but an intractable number of constraints. Constraint sampling has
been a prominent tool in ALP, and we cite a relevant result now. Let

\begin{align}
\langle\gamma_{z},r\rangle+\kappa_{z}\geq0,\hspace{0.2in} & \forall z\in\mathcal{L},\label{eq:linearconstraints}
\end{align}
be a set of linear inequalities in the variables $r\in\mathbb{R}^{k}$
indexed by an arbitrary set $\mathcal{L}$. Let $\psi$ be a probability
distribution on $\mathcal{L}$, we would like to take i.i.d. samples
from $\mathcal{L}$ to construct a set $\mathcal{W}\subseteq\mathcal{L}$
with
\[
\sup_{\left\{ r\mbox{ | }\langle\gamma_{z},r\rangle+\kappa_{z}\geq0,\,\forall z\in\mathcal{W}\right\} }\psi\left(\left\{ y\mbox{ : }\langle\gamma_{y},r\rangle+\kappa_{y}<0\right\} \right)\leq\epsilon.
\]

\begin{thm}
\cite[Theorem 2.1]{FR04} For any $\delta\in\left(0,1\right)$ and
$\epsilon\in\left(0,1\right)$, and

\[
m\geq\frac{4}{\epsilon}\left(k\,\ln\frac{12}{\epsilon}+\ln\frac{2}{\delta}\right),
\]
a set $\mathcal{W}$ of $m$ i.i.d. samples drawn from $\mathcal{L}$
according to distribution $\psi$, satisfies
\[
\sup_{\left\{ r\,\vert\,\langle\gamma_{z},r\rangle+\kappa_{z}\geq0,\,\forall z\in\mathcal{W}\right\} }\psi\left(\left\{ y\mbox{ : }\langle\gamma_{y},r\rangle+\kappa_{y}<0\right\} \right)\leq\epsilon
\]
with probability at least $1-\delta$.
\end{thm}
Thus, we can sample state-action pairs from any distribution $\psi$
on $K$ to obtain tractable relaxations of problems (\ref{ALP}) -
(\ref{ALP-2}) and (\ref{ALP_DISCOUNT}) - (\ref{ALP_DISCOUNT-2})
with probabilistic feasibility guarantees. Note that the number of
samples required is $O\left(\frac{1}{\epsilon}\ln\frac{1}{\epsilon},\ln\frac{1}{\delta}\right)$.

\subsection{Finite state and action spaces}

The development for finite state and action spaces is much simpler.
Now both problems (\ref{LP}) - (\ref{LP-4}) and (\ref{LP_dual})
- (\ref{LP_dual-2}) are usual linear programming problems with finitely
many variables and constraints. The usual linear programming duality
theory applies immediately to establish strong duality between these
two problems.

For this section, let $x$ denote an occupation measure on $K$ to
emphasize that it is finite-dimensional. Also suppose the benchmark
$Y$ has finite support $\mbox{supp}\, Y=\left\{ \eta_{1},\ldots,\eta_{q}\right\} \subset\mathbb{R}$,
so that constraint (\ref{DOMINANCE-2}) is equivalent to

\begin{align}
\mathbb{E}_{x}\left[\left(z\left(s,a\right)-\eta\right)_{-}\right]\geq\mathbb{E}\left[\left(Y-\eta\right)_{-}\right],\hspace{0.2in} & \forall\eta\in\mbox{supp}\, Y,\label{eq:average_dominance}
\end{align}
by \cite[Proposition 3.2]{Dentcheva2003}. Each expectation 
\[
\mathbb{E}_{x}\left[\left(z\left(s,a\right)-\eta\right)_{-}\right]=\sum_{\left(s,a\right)\in K}x\left(s,a\right)\left(z\left(s,a\right)-\eta\right)_{-}
\]
is a linear function of $x$.

For finite state and action spaces, the steady-state version of problem
(\ref{DOMINANCE}) - (\ref{DOMINANCE-2}) is:

\begin{align}
\max\hspace{0.2in} & \sum_{\left(s,a\right)\in K}r\left(s,a\right)x\left(s,a\right)\label{FINITE}\\
\mbox{s.t.}\hspace{0.2in} & \sum_{a\in A_{s}}x\left(j,a\right)-\sum_{s\in S}\sum_{a\in A\left(s\right)}P\left(j\mbox{ | }s,a\right)x\left(s,a\right)=0, & \forall j\in S,\label{FINITE-1}\\
 & \sum_{\left(s,a\right)\in\Psi}x\left(s,a\right)=1,\label{FINITE-2}\\
 & \mathbb{E}_{x}\left[\left(z\left(s,a\right)-\eta\right)_{-}\right]\geq\mathbb{E}\left[\left(Y-\eta\right)_{-}\right], & \forall\eta\in\mbox{supp}\, Y,\label{FINITE-3}\\
 & x\geq0.\label{FINITE-4}
\end{align}
Duality for problem (\ref{LP}) - (\ref{LP-4}) is immediate from
linear programming duality. As discussed in \cite[Chapter 8]{Put05},
the dual of the linear programming problem without the dominance constraints
is

\begin{align*}
\min\hspace{0.2in} & g\\
\mbox{s.t.}\hspace{0.2in} & g+h\left(s\right)-\sum_{j\in S}P\left(j\mbox{ | }s,\, a\right)h\left(j\right)\geq r\left(s,a\right), & \forall\left(s,a\right)\in K,\\
 & g\in\mathbb{R},\, h\in\mathbb{R}^{|s|}.
\end{align*}
The vector $h$ is interpreted as the average cost-to-go function.
To proceed with the dual for problem (\ref{LP}) - (\ref{LP-4}),
let $\lambda\in\mathbb{R}^{|\mathcal{Y}|}$ with $\lambda\geq0$ and
consider the piecewise linear increasing concave function
\[
u\left(\xi\right)=\sum_{\eta\in\mathcal{Y}}\lambda\left(\eta\right)\left(\xi-\eta\right)_{-}
\]
with breakpoints at $\eta\in\mathcal{Y}$. The above function $u\left(\xi\right)$
can be interpreted as a utility function for a risk-averse decision
maker. We define
\begin{align*}
\mathcal{U}\left(\mathcal{Y}\right)= & \mbox{cl}\,\mbox{cone}\,\left\{ \left(x-\eta\right)_{-}\mbox{ : }\eta\in\mathcal{Y}\right\} \\
= & \left\{ u\left(x\right)=\sum_{\eta\in\mathcal{Y}}\lambda\left(\eta\right)\left(x-\eta\right)_{-}\mbox{ for }\lambda\in\mathbb{R}^{|\mathcal{Y}|},\,\lambda\geq0\right\} 
\end{align*}
to be the set of all such functions. Since $\mathcal{Y}$ is assumed
to be finite, $\mathcal{U}\left(\mathcal{Y}\right)$ is a finite dimensional
set.
\begin{thm}
The dual to problem (\ref{FINITE}) - (\ref{FINITE-4}) is

\begin{align}
\min\hspace{0.2in} & g-\mathbb{E}\left[u\left(Y\right)\right]\label{FINITE_dual}\\
\mbox{s.t.}\hspace{0.2in} & r\left(s,a\right)+u\left(z\left(s,a\right)\right)\leq g+h\left(s\right)-\sum_{j\in S}P\left(j\,\vert\, s,\, a\right)h\left(j\right), & \forall\left(s,a\right)\in K,\label{FINITE_dual-1}\\
 & g\in\mathbb{R},\, h\in\mathbb{R}^{|S|},\, u\in\mathcal{U}\left(\mathcal{Y}\right).\label{FINITE_dual-2}
\end{align}
Strong duality holds between problem (\ref{FINITE}) - (\ref{FINITE-4})
and problem (\ref{FINITE_dual}) - (\ref{FINITE_dual-2}).\end{thm}
\begin{proof}
Introduce the Lagrangian

\begin{align*}
L\left(x,g,h,\lambda\right)\triangleq & \sum_{\left(s,a\right)\in K}r\left(s,a\right)x\left(s,a\right)+g\left[\sum_{\left(s,a\right)\in K}x\left(s,a\right)-1\right]\\
 & +\sum_{j\in S}h\left(j\right)\left[\sum_{a\in A\left(s\right)}x\left(j,a\right)-\sum_{s\in S}\sum_{a\in A\left(s\right)}P\left(j\mbox{ | }s,a\right)x\left(s,a\right)\right]\\
 & +\sum_{\eta\in\mathcal{Y}}\lambda\left(\eta\right)\left[\left(\sum_{\left(s,a\right)\in K}x\left(s,a\right)\left(z\left(s,a\right)-\eta\right)_{-}-\mathbb{E}\left[\left(Y-\eta\right)_{-}\right]\right)\right].
\end{align*}
Define the increasing concave function

\[
u\left(\xi\right)=\sum_{\eta\in\mathcal{Y}}\lambda\left(\eta\right)\left(\xi-\eta\right)_{-},
\]
then

\begin{align*}
 & \sum_{\eta\in\mathcal{Y}}\lambda\left(\eta\right)\left[\left(\sum_{\left(s,a\right)\in K}x\left(s,a\right)\left[\left(z\left(s,a\right)-\eta\right)_{-}\right]-\mathbb{E}\left[\left(Y-\eta\right)_{-}\right]\right)\right]\\
=\hspace{0.2in} & \sum_{\left(s,a\right)\in K}x\left(s,a\right)\left(u\left(z\left(s,a\right)\right)\right)-\mathbb{E}\left[u\left(Y\right)\right]
\end{align*}
by interchanging finite sums. So, the Lagrangian could also be written
as

\begin{align*}
L\left(x,g,h,u\right)= & \sum_{\left(s,a\right)\in K}r\left(s,a\right)x\left(s,a\right)+g\left[\sum_{\left(s,a\right)\in K}x\left(s,a\right)-1\right]\\
 & +\sum_{j\in S}h\left(j\right)\left[\sum_{a\in A\left(s\right)}x\left(j,a\right)-\sum_{s\in S}\sum_{a\in A\left(s\right)}P\left(j\mbox{ | }s,a\right)x\left(s,a\right)\right]\\
 & +\sum_{\left(s,a\right)\in K}x\left(s,a\right)u\left(z\left(s,a\right)\right)-\mathbb{E}\left[u\left(Y\right)\right],
\end{align*}
for $u\in U$. The dual to problem (\ref{FINITE-4}) - (\ref{FINITE-3})
is defined as

\[
\min_{g\in\mathbb{R},\, h\in\mathbb{R}^{|S|},\, u\in\mathcal{U}\left(\mathcal{Y}\right)}\left\{ \max_{x\geq0}L\left(x,g,h,u\right)\right\} .
\]
Rearranging the Lagrangian gives

\begin{align*}
L\left(x,g,h,u\right)= & \sum_{\left(s,a\right)\in K}x\left(s,a\right)\left[r\left(s,a\right)+g+h\left(s\right)-\sum_{j\in S}P\left(j\mbox{ | }s,\, a\right)h\left(j\right)+u\left(z\left(s,a\right)\right)\right]\\
 & -g-\mathbb{E}\left[u\left(Y\right)\right],
\end{align*}
so that the dual to problem (\ref{FINITE-4}) - (\ref{FINITE-3})
is

\begin{align*}
\min\hspace{0.2in} & -g-\mathbb{E}\left[u\left(Y\right)\right]\\
\mbox{s.t.}\hspace{0.2in} & r\left(s,a\right)+g+h\left(s\right)-\sum_{j\in S}P\left(j\mbox{ | }s,\, a\right)h\left(j\right)+u\left(z\left(s,a\right)\right)\leq0, & \forall\left(s,a\right)\in K,\\
 & g\in\mathbb{R},\, h\in\mathbb{R}^{|S|},\, u\in\mathcal{U}\left(\mathcal{Y}\right).
\end{align*}
Since $g$ and $h$ are unrestricted, take $g=-g$ and $h=-h$ to
get the desired result.
\end{proof}
We used linear programming duality in the preceding proof for illustration.
Alternatively, we could have just applied our general strong duality
result from earlier. It is immediate that problem (\ref{FINITE_dual-1})
- (\ref{FINITE_dual-2}) is the finite-dimensional version of problem
(\ref{LP_dual-6}) - (\ref{LP_dual-8}).

There is no difficulty with the Slater condition for problems (\ref{FINITE-3})
- (\ref{FINITE-4}) and (\ref{FINITE_dual-1}) - (\ref{FINITE_dual-2})
as there is in \cite{Dentcheva2003,Dentcheva2004}. In \cite{Dentcheva2003,Dentcheva2004},
the decision variable in a stochastic program is a random variable
so stochastic dominance constraints are nonlinear. In our case, the
decision variable $x$ is in the space of measures and the dominance
constraints are linear. Linear programming duality does not depend
on the Slater condition.

The development for the discounted case is similar. In terms of discounted
occupation measures $x$, problem (\ref{DISCOUNT-2}) - (\ref{DISCOUNT-1})
is

\begin{align}
\max\hspace{0.2in} & \sum_{\left(s,a\right)\in K}r\left(s,a\right)x\left(s,a\right)\label{FINITE_DISCOUNT}\\
\mbox{s.t.}\hspace{0.2in} & \sum_{a\in A_{s}}x\left(j,a\right)-\sum_{s\in S}\sum_{a\in A\left(s\right)}\gamma\, P\left(j\mbox{ | }s,a\right)x\left(s,a\right)=\alpha\left(j\right), & \forall j\in S,\label{FINITE_DISCOUNT-1}\\
 & \mathbb{E}_{x}\left[\left(z\left(s,a\right)-\eta\right)_{-}\right]\geq\mathbb{E}\left[\left(Y-\eta\right)_{-}\right], & \forall\eta\in\mathcal{Y},\label{FINITE_DISCOUNT-2}\\
 & x\geq0.\label{FINITE_DISCOUNT-3}
\end{align}
We compute the dual to problem (\ref{FINITE_DISCOUNT}) - (\ref{FINITE_DISCOUNT-3})
in the next theorem using the space of utility functions $U$ from
earlier.
\begin{thm}
The dual to problem (\ref{FINITE_DISCOUNT}) - (\ref{FINITE_DISCOUNT-3})
is

\begin{align}
\min\hspace{0.2in} & \sum_{j\in S}\alpha\left(j\right)v\left(j\right)-\mathbb{E}\left[u\left(Y\right)\right]\label{FINITE_DISCOUNT_dual}\\
\mbox{s.t.}\hspace{0.2in} & v\left(s\right)-\sum_{s\in S}\sum_{a\in A_{s}}\gamma\, P\left(j\,\vert\, s,a\right)v\left(j\right)\geq r\left(s,a\right)+u\left(z\left(s,a\right)\right), & \forall\left(s,a\right)\in K,\label{FINITE_DISCOUNT_dual-1}\\
 & v\in\mathbb{R}^{|S|},\, u\in\mathcal{U}\left(\mathcal{Y}\right).\label{FINITE_DISCOUNT_dual-2}
\end{align}
Strong duality holds between problem (\ref{FINITE_DISCOUNT}) - (\ref{FINITE_DISCOUNT-3})
and problem (\ref{FINITE_DISCOUNT_dual}) - (\ref{FINITE_DISCOUNT_dual-2}).
\end{thm}

\section{Portfolio optimization}

We use an infinite horizon discounted portfolio optimization problem
to illustrate our ideas in this section. A single period portfolio
optimization with stochastic dominance constraints is analyzed in
\cite{Dentcheva06}. Specifically, the model in \cite{Dentcheva06}
puts a stochastic dominance constraint on the return rate of a portfolio
allocation. We use this model as our motivation for the dynamic setting
and put a stochastic dominance constraint on the discounted infinite
horizon return rate.

Suppose there are $n$ assets whose prices evolve according to a discrete
time Markov chain. We can include a risk-less asset with a constant
return rate in this set. The asset prices at time $t$ are

\[
p_{t}=\left(p_{t}\left(1\right),\ldots,p_{t}\left(n\right)\right)\in\mathbb{R}^{n},
\]
where $ $$p_{t}\left(i\right)$ is the price per share of asset $i$
at time $t$. The portfolio at time $t$ is captured by
\[
x_{t}=\left(x_{t}\left(1\right),\ldots,x_{t}\left(n\right)\right)\in\mathbb{R}^{n},
\]
where $x_{t}\left(i\right)$ is the quantity of shares held of asset
$i$ at time $t$. For a cleaner model, we just treat each $x_{t}\left(i\right)$
as a continuous decision variable. We require $\sum_{i=1}^{n}x_{t}\left(i\right)=1$
and $x_{t}\geq0$ for all $t\geq0$, there is no shorting. The total
wealth at time $t$ is then $\langle p_{t},x_{t}\rangle$.

At each time $t\geq0$, the investor observes the current prices of
the assets and then updates portfolio positions subject to transaction
costs before new prices are realized. Let $a_{t}\subset\mathbb{R}^{n}$
be the buying and selling decisions at time $t$, where $a_{t}\left(i\right)$
is the total change in the number of shares held of asset $i$. Define

\begin{align*}
A\left(p,x\right)\triangleq & \left\{ a\in\mathbb{R}^{n}\mbox{ : }x\left(i\right)+a\left(i\right)\geq0\mbox{ for all }i=1,\ldots,n,\right.\\
 & \left.\sum_{i=1}^{n}p\left(i\right)a\left(i\right)=0\right\} ,
\end{align*}
to be the set of feasible reallocations given prices and holdings
$x$. The constraint $\sum_{i=1}^{n}p\left(i\right)a\left(i\right)=0$
requires the total change in wealth from buying and selling decisions
to be zero in any period. The system dynamic for portfolio positions
is then

\begin{align}
x_{t}\left(t+1\right)=x_{t}\left(i\right)+a_{t}\left(i\right),\hspace{0.2in} & i=1,\ldots,n,\, t\geq0.\label{PORTFOLIO}
\end{align}
The transaction costs $c\mbox{ : }A\rightarrow\mathbb{R}$ are defined
to be
\[
c\left(a\right)\triangleq\sum_{i=1}^{n}a_{t}\left(i\right)^{2},
\]
this cost function is a moment on $S\times A$.

The overall return rate between time $t$ and $t+1$ is
\[
z\left(p_{t},x_{t};\, p_{t+1},x_{t+1}\right)\triangleq\frac{\langle p_{t+1},x_{t+1}\rangle-\langle p_{t},x_{t}\rangle}{\langle p_{t},x_{t}\rangle}.
\]
We make the reasonable assumption that $z\left(p_{t},x_{t};\, p_{t+1},x_{t+1}\right)$
is bounded for this example.

We want to minimize discounted transaction costs
\[
C\left(\pi,\nu\right)\triangleq\mathbb{E}_{\nu}^{\pi}\left[\sum_{t\geq0}\delta^{t}c\left(a_{t}\right)\right]
\]
subject to a stochastic dominance constraint on the discounted return
rate. Define

\[
Z_{\eta}\left(\pi,\nu\right)\triangleq\mathbb{E}_{\nu}^{\pi}\left[\sum_{t=0}^{\infty}\delta^{t}\left(z\left(p_{t},x_{t};\, p_{t+1},x_{t+1}\right)-\eta\right)_{-}\right]
\]
to be the expected discounted shortfall in relative returns at level
$\eta$. We introduce a benchmark $Y$ for the discounted return rate,
and we suppose the support of $Y$ is bounded within $\left[a,b\right]$.
In this example, the benchmark can be taken as any market index.

We absorb the system dynamic (\ref{PORTFOLIO}) into a transition
kernel $Q$. Our resulting portfolio optimization problem is then
\begin{align}
\max_{\pi\in\Pi}\hspace{0.2in} & -C\left(\pi,\nu\right)\label{PORTFOLIO-1}\\
\mbox{s.t.}\hspace{0.2in} & Z_{\eta}\left(\pi,\nu\right)\geq\mathbb{E}\left[\left(Y-\eta\right)_{-}\right], & \eta\in\left[a,b\right].\label{PORTFOLIO-2}
\end{align}
In the linear programming formulation of (\ref{PORTFOLIO-1}) - (\ref{PORTFOLIO-2}),
we simply augment the state space and consider occupation measures
over sequences

\[
\left(p_{t},x_{t},a_{t};\, p_{t+1},x_{t+1},a_{t+1}\right)
\]
to correctly compute $z$.

\section{Conclusion}

We have shown how to use stochastic dominance constraints in infinite
horizon MDPs. Convex analytic methods establish that stochastic dominance
constrained MDPs can be solved via linear programming, and have corresponding
dual linear programming problems. Conditions are given for strong
duality to hold between these two linear programs. Utility functions
appear in the dual as pricing variables corresponding to the stochastic
dominance constraints. This result has intuitive appeal, since our
stochastic dominance constraints are defined in terms of utility functions,
and parallels earlier results \cite{Dentcheva2003,Dentcheva2004,Dentcheva2008}.
Our results are shown to be extendable to many types of stochastic
dominance constraints, particularly multivariate ones.

There are three main directions for our future work. First, we will
consider efficient strategies for computing the optimal policy to
stochastic dominance constrained MDPs. Second, we would like explore
other methods for modeling risk in MDPs using convex analytic methods.
Specifically, we are interested in solving MDPs with convex risk measures
and chance constraints with ``static'' optimization problems as
we have done here. Third, as suggested by the portfolio example, we
will consider online data-driven optimization for the stochastic dominance-constrained
MDPs in this paper. The transition probabilities of underlying MDPs
are not known in practice and must be learned online.

\bibliographystyle{plain}
\bibliography{sdcmdp}

\end{document}